\documentclass{amsart}
\usepackage{mathrsfs}
\usepackage{cases}
\usepackage{amssymb,amscd,amsthm}
\usepackage{latexsym}
\usepackage{hyperref}
\usepackage{color}
\date{\today}

\newcommand{\Z}{{\mathbb Z}}
\newcommand{\R}{{\mathbb R}}
\newcommand{\C}{{\mathbb C}}
\newcommand{\E}{{\mathbb E}}

\newcommand{\PP}{{\mathbb P}}

\newtheorem{theorem}{Theorem} [section]
\newtheorem{remark}[theorem]{Remark}
\newtheorem{lemma}[theorem]{Lemma}
\newtheorem{prop}[theorem]{Proposition}
\newtheorem{coro}[theorem]{Corollary}
\newtheorem{definition}[theorem]{Definition}

\begin{document}

\title{A Short Course on One-Dimensional Random Schr\"odinger Operators}

\author[D.\ Damanik]{David Damanik}

\address{Department of Mathematics, Rice University, Houston, TX~77005, USA}

\email {\href{mailto:damanik@rice.edu}{damanik@rice.edu}}

\urladdr {\href{http://www.ruf.rice.edu/~dtd3}{www.ruf.rice.edu/$\sim$dtd3}}

\thanks{D.\ D.\ was supported in part by NSF grant DMS--0800100.}

\begin{abstract}
We discuss various approaches to localization results for one-dimensional random Schr\"odinger operators, both discrete and continuum. We focus in particular on the approach based on F\"urstenberg's Theorem and the Kunz-Souillard method. These notes are based on a series of five one-hour lectures given at University College London in June/July 2011.
\end{abstract}

\maketitle

\tableofcontents

\section{Introduction}

\subsection{Overview}

This short course will discuss Schr\"odinger operators with random potentials, both in the discrete and in the continuum setting. Thus, we will consider operators
$$
[H_\omega \psi](n) = [ (\Delta + V_\omega) \psi](n) = \sum_{|m - n|_1 = 1} \psi(m) + V_\omega(n) \psi(n)
$$
in $\ell^2(\Z^d)$ and operators
$$
[H_\omega \psi](x) = [ (-\Delta + V_\omega) \psi](x) = - \sum_{j = 1}^d \partial_j^2 \psi (x) + V_\omega(x) \psi(x)
$$
in $L^2(\R^d)$.\footnote{There is a tradition to consider $-\Delta$ in the continuum case and $\Delta$ in the discrete case. In fact, the definition of the discrete Laplacian used here deviates from what a standard discretization procedure would yield in that it drops a multiple of the identity, resulting in a spectrum that is symmetric about zero. These choices are standard and do not affect the mathematical content of the results obtained, but merely their formulation.} Here, $V_\omega$ is a random potential to be specified below. Operator domains are chosen appropriately to ensure self-adjointness. The goal is then to identify the typical spectral properties of such operators, that is, spectral properties that hold with probability one.

Let us first describe the global expected picture. In dimensions $d = 1$ or $2$, it is expected that the associated quantum system is localized. On a spectral level, this manifests itself in so-called \emph{spectral localization}, that is, the operator $H_\omega$ has a complete set of eigenvectors that decay exponentially in space. On a quantum dynamical level, one expects \emph{dynamical localization}, that is, for initially localized $\psi$, $e^{-itH_\omega} \psi$ remains uniformly localized in space for all times. The dependence on $\omega$ of such statements may be controlled in some non-trivial way. In dimensions $d \ge 3$, on the other hand, localization phenomena are expected to depend on the strength of the randomness and the energy regime. More explicitly, spectral measures are expected to be pure point near the boundary of the spectrum (with associated exponential decay of the corresponding eigenvectors) and dynamical localization is expected to hold for initial states whose spectral measures are supported in a suitable neighborhood of the boundary of the spectrum. The size of this localized neighborhood of the boundary of the spectrum is expected to depend on the strength of the randomness, and in fact \emph{delocalization} takes place outside a suitable neighborhood of the boundary of the spectrum. Here, delocalization refers to continuous spectral measures and spreading of $e^{-itH_\omega} \psi$ in space as $|t| \to \infty$.

The localization side of the expected global picture is quite well understood. The one glaring gap in our understanding is that localization at all energies is still an open problem in the case $d = 2$. On the other hand, it is well understood that there is localization at all energies in the case $d = 1$ and that there is localization near the boundary of the spectrum in the case $d \ge 2$.\footnote{An important exception is the Bernoulli single-site distribution in the \emph{discrete} case.} Moreover, for each of these statements, there are multiple proofs, each of which sheds a different light on the localization phenomenon and works under different assumptions on the model. In that sense, there is no single ``best'' localization proof.

The cases $d = 1$ and $d \ge 2$ are quite different. For the purpose of a short course, one needs to focus on one case and hence the paths need to diverge at this point. We have decided to focus on the case $d = 1$ in this short course. One of the reasons is that there exist graduate-level expositions of localization proofs in the case $d \ge 2$ similar in style and scope to the one offered here \cite{kirsch, klein, stolz}, but not in the case $d = 1$.

\subsection{Models and Results}

Let us now describe the models we consider precisely. As explained above, we will restrict our attention to the one-dimensional scenario and hence consider operators
\begin{equation}\label{e.operd}
[H_\omega \psi](n) = \psi(n+1) + \psi(n-1) + V_\omega(n) \psi(n)
\end{equation}
in $\ell^2(\Z)$ and operators
\begin{equation}\label{e.operc}
[H_\omega \psi](x) = - \psi'' (x) + V_\omega(x) \psi(x)
\end{equation}
in $L^2(\R)$. The potential $V_\omega$ will depend on a parameter $\omega$, chosen from a probability space $(\Omega,\mu)$.

Before giving a formal definition of $\Omega$, $\mu$, and $V_\omega$ for $\omega \in \Omega$, let us informally discuss what the model is supposed to describe. The type of randomness we want to capture is that of a sequence of independent identically distributed (``i.i.d.'') random variables. That is, in the discrete setting for example, the potential value at each site, $V_\omega(n)$, is supposed to be drawn from some probability distribution on $\R$ (since we want our potentials to be real-values to ensure self-adjointness of $H_\omega$), this probability distribution is the same for all sites, and the values are drawn independently from this distribution. Let us call this common probability distribution $\nu$. Clearly, the resulting potentials will be bounded if and only if the topological support of $\nu$ is compact. All interesting phenomena related to random potentials already occur in this special case, so we will restrict our attention to it. Once $\nu$ is chosen, it is then clear how to formally define the $V_\omega$'s:

\begin{definition}[Discrete Case]\label{d.potd}
Given a probability measure $\nu$ on $\R$ with $\mathrm{supp} \, \nu$ compact, let $\Omega = (\mathrm{supp} \, \nu)^\Z$,  $\mu = \nu^\Z$, and $V_\omega(n) = \omega_n$ for $\omega \in \Omega$.
\end{definition}

By compactness of $\mathrm{supp} \, \nu$, the potentials $V_\omega$ are (real-valued and) bounded, and hence the operators $H_\omega$ in \eqref{e.operd} with domain given by $\ell^2(\Z)$ are bounded and self-adjoint.

\begin{remark}
Let us note for later use the following regarding Definition~\ref{d.potd}. Since $\mathrm{supp} \, \nu$ is assumed to be compact, $\Omega$ is a compact metric space. Moreover, with the homeomorphism $T : \Omega \to \Omega$, $(T\omega)_n = \omega_{n+1}$ and the continuous sampling function $f : \Omega \to \R$, $f(\omega) = \omega_0$, we can write $V_\omega(n) = f(T^n \omega)$.
\end{remark}

In the continuum case, we follow a similar path and construct the potentials $V_\omega$ from a sequence of i.i.d.\ random variables. Each of them will play the role of a coupling constant, signifying the local strength of a common single-site potential. To ensure independence on the level of the potential, one assumes that the support of the single-site potential is contained in a sufficiently small interval. Typically, one wants non-overlap. Thus, we are led to the following definition of the potentials in \eqref{e.operc}:

\begin{definition}[Continuum Case]\label{d.potc}
Given a probability measure $\nu$ on $\R$ with $\mathrm{supp} \, \nu$ compact and $v \in L^1(0,1)$ real-valued and non-zero {\rm (}in $L^1(0,1)${\rm )}, let $\Omega = (\mathrm{supp} \, \nu)^\Z$,  $\mu = \nu^\Z$, and
$$
V_\omega(x) = \sum_{n \in \Z} \omega_n v(x-n)
$$
for $\omega \in \Omega$.
\end{definition}

That is, the potential $V_\omega$ is given on the interval $(n,n+1)$ by the coupling constant $\omega_n$ times the single-site potential $v$ shifted from $(0,1)$ to $(n,n+1)$. Consequently, $V_\omega$ is uniformly $L^1_\mathrm{loc}$ and hence $H_\omega$ as defined in \eqref{e.operc} is self-adjoint on a suitable domain by form methods or Sturm-Liouville theory.

Let us now state the main results for these operators. We begin with spectral localization in the discrete case.

\begin{theorem}[Spectral Localization - Discrete Case]\label{t.speclocd}
Suppose the potentials $\{ V_\omega \}$ are as defined in Definition~\ref{d.potd} and the operators $\{ H_\omega \}$ are given by \eqref{e.operd}. Assume in addition that $\mathrm{supp} \, \nu$ contains at least two points. Then, for $\mu$-almost every $\omega \in \Omega$, the operator $H_\omega$ has pure point spectrum with exponentially decaying eigenvectors.
\end{theorem}

\begin{remark} {\rm (a)} The additional assumption is necessary. Indeed, if $\mathrm{supp} \, \nu$ consists of a single point $a$, then the resulting potentials are constant and the operators $H_\omega$ have purely absolutely continuous spectrum. In fact, they are just the discrete Laplacian up to a constant shift in energy.
\\[1mm]
{\rm (b)} The conclusion of the theorem means that for $\mu$-almost every $\omega$, there are $E_k(\omega) \in \R$ and  $\psi_k(\omega) \in \ell^2(\Z)$, $k \in \Z_+$, such that $H_\omega \psi_k(\omega) = E_k(\omega) \psi_k(\omega)$ for every $k \in \Z_+$, the finite linear combinations of the $\psi_k(\omega)$ are dense in $\ell^2(\Z)$, and we have exponential decay of the form
\begin{equation}\label{e.expevdecd}
|\psi_k(\omega,n)| \le C_k(\omega) e^{-\gamma_k(\omega) | n - n_k(\omega) |}
\end{equation}
with suitable $C_k(\omega) , \gamma_k(\omega) \in (0,\infty)$ and $n_k(\omega) \in \Z$.
\\[1mm]
{\rm (c)} One usually has some control over $C_k(\omega) , \gamma_k(\omega) \in (0,\infty)$ and $n_k(\omega) \in \Z$. In fact, having control of this kind is helpful in establishing dynamical localization.
\end{remark}

The formulation of the result in the continuum case is completely analogous:

\begin{theorem}[Spectral Localization - Continuum Case]\label{t.speclocc}
Suppose the potentials $\{ V_\omega \}$ are as defined in Definition~\ref{d.potc} and the operators $\{ H_\omega \}$ are given by \eqref{e.operc}. Assume in addition that $\mathrm{supp} \, \nu$ contains at least two points. Then, for $\mu$-almost every $\omega \in \Omega$, the operator $H_\omega$ has pure point spectrum with exponentially decaying eigenvectors.
\end{theorem}

Here, exponential decay means that the eigenvectors obey an estimate of the form
\begin{equation}\label{e.expevdecc}
|\psi_k(\omega,x)| \le C_k(\omega) e^{-\gamma_k(\omega) | x - x_k(\omega) |}
\end{equation}
with suitable $C_k(\omega) , \gamma_k(\omega) \in (0,\infty)$ and $x_k(\omega) \in \R$. While the statement of the result is completely analogous, the various existing proofs are significantly more involved in the continuum case as we will see later in the course. Moreover, there is a hidden subtlety in the continuum case. While the Lyapunov exponents are always positive at all energies in the discrete case, they may vanish on a discrete set of exceptional energies in the continuum case. This difference does not manifest itself on the spectral level, at least not in the way spectral localization is formulated above, but it does affect dynamical localization properties. There may in fact be non-trivial transport in the continuum case for initial states whose spectral measures give non-trivial weight to the set of exceptional energies.

\section{Lyapunov Exponents}

\subsection{Definition of the Lyapunov Exponent}

In this section we discuss the one-parameter family $\{ \gamma(E) \}_{E \in \R}$ of Lyapunov exponents associated with a given operator family $\{ H_\omega \}_{\omega \in \Omega}$. These exponents measure the typical exponential growth or decay of the solutions of the associated difference (resp., differential) equation, that is,
\begin{equation}\label{e.eved}
u(n+1) + u(n-1) + V_\omega(n) u(n) = E u(n)
\end{equation}
in the discrete case and
\begin{equation}\label{e.evec}
-u''(x) + V_\omega(x) u(x) = E u(x)
\end{equation}
in the continuum case. It is convenient to introduce the transfer matrices $M_{E,\omega}(\cdot)$ and measure their norms.

Let us discuss the discrete case first. Note that a solution of \eqref{e.eved} is uniquely determined by fixing two consecutive values. Moreover, the map from a given pair of consecutive values to another pair is clearly linear and hence given by a $2 \times 2$ matrix. In particular, for each $n \in \Z$, there is a $2 \times 2$ matrix $M_{E,\omega}(n)$, called the \emph{transfer matrix}, so that for any solution $u$ of \eqref{e.eved}, we have
$$
\begin{pmatrix} u(n+1) \\ u(n) \end{pmatrix} = M_{E,\omega}(n) \begin{pmatrix} u(1) \\ u(0) \end{pmatrix}.
$$
In fact, it is easy to give an explicit description of $M_{E,\omega}(n)$. Two popular and useful ways of viewing the transfer matrix are the following. The first follows directly from the way we introduced this matrix: denote by $u_D$ (resp., $u_N$) the solution of \eqref{e.eved} that obeys $(u(1),u(0))^T = (1,0)$ (resp., $(u(1),u(0))^T = (0,1)$). These are the Dirichlet and Neumann solution, respectively. Then,
\begin{equation}\label{e.tm1}
M_{E,\omega}(n) = \begin{pmatrix} u_D(n+1) & u_N(n+1) \\ u_D(n) & u_N(n) \end{pmatrix}.
\end{equation}
The second builds up the matrix in a step-by-step way. Notice that \eqref{e.eved} is equivalent to
\begin{equation}\label{e.tm2}
\begin{pmatrix} u(n+1) \\ u(n) \end{pmatrix} = \begin{pmatrix} E - V_\omega(n) & -1 \\ 1 & 0 \end{pmatrix} \begin{pmatrix} u(n) \\ u(n-1) \end{pmatrix}.
\end{equation}
For $n \ge 1$, iterating \eqref{e.tm2}, we find
\begin{equation}\label{e.tm3}
M_{E,\omega}(n) = \begin{pmatrix} E - V_\omega(n) & -1 \\ 1 & 0 \end{pmatrix} \times \cdots \times \begin{pmatrix} E - V_\omega(1) & -1 \\ 1 & 0 \end{pmatrix}.
\end{equation}
For $n \le -1$, there is an analogous formula for $M_{E,\omega}(n)$. By definition, $\det M_{E,\omega}(n) = 1$ and hence $\| M_{E,\omega}(n) \| \ge 1$, so that $\log \| M_{E,\omega}(n) \| \ge 0$. Lyapunov exponents measure the asymptotic behavior of $\frac1n \log \| M_{E,\omega}(n) \|$ and hence the rate of exponential growth of $\| M_{E,\omega}(n) \|$.

\begin{definition}\label{d.lyapbeh}
If the limit
$$
\lim_{n \to \infty} \frac1n \log \| M_{E,\omega}(n) \|
$$
exists, we denote it by $\gamma_\omega(E)$ and say that $H_\omega$ has Lyapunov behavior at energy $E$.
\end{definition}

Later in this section, we will pursue the following goals. First, we want to show that in our setting, for every energy $E$, $\gamma_\omega(E)$ exists for $\mu$-almost every $\omega \in \Omega$ and is equal to a constant, which we denote by $\gamma(E)$. As observed above, we always have $\gamma_\omega(E) , \gamma(E) \ge 0$. Hence, the main issue is whether these exponents are strictly positive. Our second goal will be to exhibit exponential growth or decay for all solutions of \eqref{e.eved} provided that $\gamma_\omega(E)$. Third, we wish to show that randomness forces positive exponents. In the discrete case we will indeed prove that $\gamma(E) > 0$ for every $E$. Naively, this seems to already establish spectral localization. While this conclusion cannot be drawn in general, it is true that global positivity strongly hints at spectral localization and, with more work, will in fact imply it for random potentials. We will discuss this in more detail later.

Let us now discuss the continuum case. A solution of \eqref{e.evec} is uniquely determined by its value and its derivative at some point. The defining property of the transfer matrices is again that it maps solution data from the origin to a given point, that is,
$$
\begin{pmatrix} u'(x) \\ u(x) \end{pmatrix} = M_{E,\omega}(x) \begin{pmatrix} u'(0) \\ u(0) \end{pmatrix}.
$$
Introducing Dirichlet and Neumann solutions as above, we can write
\begin{equation}\label{e.tm4}
M_{E,\omega}(x) = \begin{pmatrix} u_D'(x) & u_N'(x) \\ u_D(x) & u_N(x) \end{pmatrix}.
\end{equation}
On the other hand, \eqref{e.evec} is equivalent to
\begin{equation}\label{e.tm5}
\begin{pmatrix} u'(x) \\ u(x) \end{pmatrix}' = \begin{pmatrix} 0 & V_\omega(n) - E \\ 1 & 0 \end{pmatrix} \begin{pmatrix} u'(x) \\ u(x) \end{pmatrix}.
\end{equation}
Thus, using \eqref{e.tm5} to trace the evolution of the Dirichlet solution $u_D$ and the Neumann solution $u_N$ and plugging this into \eqref{e.tm4}, we find
\begin{equation}\label{e.tm6}
\frac{\partial}{\partial x} M_{E,\omega}(x) = \begin{pmatrix} 0 & V_\omega(n) - E \\ 1 & 0 \end{pmatrix} M_{E,\omega}(x),
\end{equation}
which is the continuum analogue of \eqref{e.tm2}. The definition of Lyapunov behavior and the exponent $\gamma_\omega(E)$ is then the same as in Definition~\ref{d.lyapbeh}.

\subsection{Existence of the Lyapunov Exponent}

For the existence of the Lyapunov exponent, it is important that $(\Omega,\mu,T)$ is ergodic, that is, if $B$ belongs to the Borel $\sigma$-algebra and satisfies $T^{-1} B = B$, then either $\mu(B) = 0$ or $\mu(B) = 1$. Let us establish this fact first:

\begin{lemma}\label{l.fserg}
$(\Omega,\mu,T)$ is ergodic.
\end{lemma}

\begin{proof}
Suppose $B$ belongs to the Borel $\sigma$-algebra and satisfies $T^{_1} B = B$. Let $\varepsilon > 0$. By construction of the product measure, we can find a finite collection $C_1 , \ldots , C_k$ of mutually disjoint cylinder sets such that
\begin{equation}\label{e.fserg0}
\mu \left( B \bigtriangleup C \right) < \varepsilon, \quad \text{ where } C =  \bigcup_{j=1}^k C_j.
\end{equation}
Choose $n$ so that $T^{-n} C$ and $C$ depend on disjoint sets of coordinates. Thus, by construction of the product measure and its $T$-invariance, we have
\begin{equation}\label{e.fserg1}
\mu( T^{-n} C \cap C ) = \mu(T^{-n} C) \mu(C) = \mu(C) \mu(C) = \mu(C)^2.
\end{equation}
We also have by the choice of $B$, the $T$-invariance of $\mu$, and the choice of $C$,
$$
\mu(B \bigtriangleup T^{-n} C) = \mu(T^{-n} B \bigtriangleup T^{-n} C) = \mu(T^{-n} (B \bigtriangleup C)) = \mu(B \bigtriangleup C) < \varepsilon.
$$
Thus,
$$
\mu(B \bigtriangleup (C \cap T^{-n} C)) \le \mu ( B \bigtriangleup C ) + \mu ( B \bigtriangleup T^{-n} C ) < 2 \varepsilon,
$$
where the first step follows from $B \bigtriangleup (C \cap T^{-n} C) \subseteq B \bigtriangleup C \cup B \bigtriangleup T^{-n} C$. It follows that
\begin{equation}\label{e.fserg2}
\left| \mu ( B ) - \mu (C \cap T^{-n} C) \right| < 2 \varepsilon.
\end{equation}
Hence,
\begin{align*}
\left| \mu ( B ) - \mu ( B )^2 \right| & = \left| \mu ( B ) - \mu (C \cap T^{-n} C) \right|  + \left| \mu (C \cap T^{-n} C) - \mu ( B )^2 \right| \\
& < 2 \varepsilon + \left| \mu (C)^2 - \mu ( B )^2 \right| \\
& = 2 \varepsilon + \left| \mu (C) + \mu ( B ) \right| \cdot \left| \mu (C) - \mu ( B ) \right| \\
& \le 2 \varepsilon + 2 \varepsilon \\
& = 4 \varepsilon,
\end{align*}
where we used \eqref{e.fserg1} and \eqref{e.fserg2} in the second step and \eqref{e.fserg0} in the fourth step. Since $\varepsilon > 0$ was arbitrary, we have $\mu ( B ) = \mu ( B )^2$ and hence $\mu ( B ) \in \{ 0,1 \}$.
\end{proof}

A useful characterization of ergodicity is the following:

\begin{lemma}\label{l.ergchar}
$(\Omega,\mu,T)$ is ergodic if and only if every measurable $T$-invariant function on $\Omega$ is $\mu$-almost surely constant.
\end{lemma}

\begin{proof}
Exercise.
\end{proof}

As an aside we mention that the spectrum of $H_\omega$ is $\mu$-almost surely constant.

\begin{prop}
There is a set $\Sigma \subset \R$ such that $\sigma(H_\omega) = \Sigma$ for $\mu$-almost every $\omega \in \Omega$.
\end{prop}

\begin{proof}
By definition of $H_\omega$, we have for every bounded open interval $I = (a,b)$,
$$
\mathrm{tr} \, \chi_I(H_{T\omega}) = \sum_{n \in \Z} \langle \delta_n,
\chi_I(H_{T\omega}) \delta_n \rangle = \sum_{n \in \Z} \langle
\delta_{n+1}, \chi_I(H_{\omega}) \delta_{n+1} \rangle = \mathrm{tr} \, \chi_I(H_{\omega}).
$$
Thus, by Lemmas~\ref{l.fserg} and \ref{l.ergchar}, $\mathrm{tr} \, \chi_I(H_{\omega})$ is $\mu$-almost surely constant. Of course, the full measure set depends on $I$. However, if we restrict our attention to intervals with rational endpoint, we can choose a common full measure set. Since $\mathrm{tr} \, \chi_I(H_{\omega}) = 0$ if and only if $\sigma(H_\omega) \cap I = \emptyset$, the proposition follows.
\end{proof}

In fact, one can determine $\Sigma$ explicitly:

\begin{prop}
With $\Sigma = [-2,2] + \mathrm{supp} \, \nu$, we have $\sigma(H_\omega) = \Sigma$ for $\mu$-almost every $\omega \in \Omega$.
\end{prop}

\begin{proof}
Exercise.
\end{proof}

This finishes the aside on the almost sure spectrum. Let us return to our study of Lyapunov exponents and to this end recall Kingman's Subadditive Ergodic Theorem:

\begin{theorem}\label{t.kingman}
Suppose $(\Omega,\mu,T)$ is ergodic. If $g_n : \Omega \to \R$ are measurable, obey $\|g_n\|_\infty \lesssim n$ and the subadditivity condition
$$
g_{n+m}(\omega) \le g_n(\omega) + g_m(T^n \omega),
$$
then
$$
\lim_{n \to \infty} \frac{1}{n} g_n(\omega) = \inf_{n \ge 1} \,
\frac{1}{n} \int g_n (\omega) \, d\mu(\omega)
$$
for $\mu$-almost every $\omega \in \Omega$.
\end{theorem}

\begin{coro}
For every $E$, there is $\gamma (E) \ge 0$ such that for $\mu$-almost every $\omega \in \Omega$, we have
$$
\lim_{n \to \infty} \frac1n \log \| M_{E,\omega}(n) \| = \gamma(E).
$$
In particular, for these $\omega$'s, $H_\omega$ has Lyapunov behavior at energy $E$ and $\gamma_\omega(E) = \gamma(E)$.
\end{coro}

\begin{proof}
We apply Theorem~\ref{t.kingman}. Fix $E$ and set
$$
g_n(\omega,E) = \log \| M_{E,\omega}(n) \|.
$$
It is easy to check that these functions satisfy the subadditivity condition
$$
g_{n+m}(\omega,E) \le g_n(\omega,E) + g_m(T^n \omega , E).
$$
We may therefore apply Kingman's Subadditive Ergodic Theorem. For $\mu$-almost every $\omega \in \Omega$, we have
$$
\lim_{n \to \infty} \frac{1}{n} g_n(\omega,E) = \inf_{n \ge 1} \, \frac{1}{n} \int g_n (\omega,E) \, d\mu(\omega),
$$
so we denote the right-hand side by $\gamma(E)$ and the corollary follows.
\end{proof}

\subsection{The Oseledec-Ruelle Theorem}

\begin{theorem}\label{ruelle}
Suppose $A_n \in \mathrm{SL}(2,\R)$ obey
$$
\lim_{n \to \infty} \frac{1}{n} \| A_n \| = 0
$$
and
$$
\lim_{n \to \infty} \frac{1}{n} \log \| A_n \cdots A_1 \| = \gamma
> 0.
$$
Then there exists a one-dimensional subspace $V \subset \R^2$ such
that
$$
\lim_{n \to \infty} \frac{1}{n} \log \| A_n \cdots A_1 v \| = -
\gamma \quad \text{ for } v \in V \setminus \{0\}
$$
and
$$
\lim_{n \to \infty} \frac{1}{n} \log \| A_n \cdots A_1 v \| =
\gamma \quad \text{ for } v \not\in V.
$$
\end{theorem}

\begin{proof}
Write $T_n = A_n \cdots A_1$, $t_n = \| T_n \|$, $a_n = \| A_n
\|$, and
$$
u_\theta = \left( \begin{array}{c} \cos \theta \\ \sin \theta
\end{array} \right).
$$
Since $|T_n|$ is self-adjoint and unimodular, it has eigenvalues
$t_n$ and $t_n^{-1}$. Define $\theta_n$ by $|T_n| u_{\theta_n} =
t_n^{-1} u_{\theta_n}$. Then, by self-adjointness again, $|T_n|
u_{\theta_n + \frac{\pi}{2}} = t_n u_{\theta_n + \frac{\pi}{2}}$.
It follows from the trigonometric formulae $\sin (x + y) = \sin x
\cos y + \cos x \sin y$ and $\cos (x + y) = \cos x \cos y - \sin x
\sin y$ that
$$
u_\theta = \cos (\theta - \theta_n) u_{\theta_n} + \sin (\theta -
\theta_n) u_{\theta_n + \frac{\pi}{2}}.
$$
Thus, using $\| T_n u_\theta \| = \| |T_n| u_\theta \|$,
\begin{equation}\label{tutheta}
\| T_n u_\theta \|^2 = t_n^2 \sin^2 (\theta - \theta_n) + t_n^{-2}
\cos^2 (\theta - \theta_n).
\end{equation}
By \eqref{tutheta} (with $n$ replaced by $n+1$),
\begin{align*}
t_{n+1}^2 \sin^2 (\theta_n - \theta_{n+1}) & \le \| T_{n+1} u_{\theta_n} \|^2 \\
& \le a_{n+1}^2 \| T_n u_{\theta_n} \|^2 \\
& = a_{n+1}^2 t_n^{-2}.
\end{align*}
Since $A_{n+1}$ is unimodular,
$$
t_n = \| T_n \| \le \| T_{n+1} \| \, \| A_{n+1}^{-1} \| = \|
T_{n+1} \| \, \| A_{n+1} \| = t_{n+1} a_{n+1},
$$
and hence
$$
t_n^2 \sin^2 (\theta_n - \theta_{n+1}) \le a_{n+1}^4 t_n^{-2}.
$$
Since $\sin^2 (x) \gtrsim x^2$ for the values of $x$ in question,
we obtain
\begin{equation}\label{thetannpo}
| \theta_n - \theta_{n+1} | \lesssim \frac{a_{n+1}^2}{t_n^2}.
\end{equation}
It follows from our assumptions that $\sum | \theta_n -
\theta_{n+1} | < \infty$ and hence $\theta_n \to \theta_{\infty}$,
which obeys
\begin{equation}\label{thetainf}
| \theta_n - \theta_\infty | \lesssim \sum_{m = n}^\infty
\frac{a_{m+1}^2}{t_m^2}.
\end{equation}
Let $u_\infty = u_{\theta_\infty}$ and $v_\infty =
u_{\theta_\infty + \frac{\pi}{2}}$. We claim that the assertion of
the theorem holds with $V$ given by the span of $u_\infty$. Since
$\| T_n v_\infty \| \le t_n$ and $\| T_n u_\infty \| \ge
t_n^{-1}$, it suffices to show
\begin{equation}\label{vinflow}
\| T_n v_\infty \|^2 \ge \tfrac12 t_n^2 \quad \text{for $n$ large
enough}
\end{equation}
and
\begin{equation}\label{uinfupp}
\limsup_{n \to \infty} \frac1n \log \| T_n u_\infty \| \le
-\gamma.
\end{equation}
Since $\theta_n - \theta_\infty \to 0$, \eqref{vinflow} follows
from \eqref{tutheta}. For every $\varepsilon > 0$ and $n$ large,
we have by \eqref{tutheta}, \eqref{thetannpo}, and
\eqref{thetainf},
\begin{align*}
\| T_n u_\infty \|^2 & = t_n^2 \sin^2 (\theta_\infty - \theta_n) +
t_n^{-2} \cos^2
(\theta_\infty - \theta_n) \\
& \le t_n^2 | \theta_\infty - \theta_n |^2 + e^{-2(1-\varepsilon)\gamma n} \\
& \lesssim t_n^2 \left( \sum_{m = n}^\infty
\frac{a_{m+1}^2}{t_m^2} \right)^2 +
e^{-2(1-\varepsilon)\gamma n} \\
& \le e^{2(1+\varepsilon ) \gamma n} \left( \sum_{m = n}^\infty
\frac{e^{2\varepsilon
\gamma m}}{e^{2(1-\varepsilon )\gamma m}} \right)^2 + e^{-2(1-\varepsilon)\gamma n} \\
& = e^{2(1+\varepsilon ) \gamma n} \left( \sum_{m = n}^\infty
e^{(4\varepsilon - 2) \gamma m} \right)^2 +
e^{-2(1-\varepsilon)\gamma n}
\end{align*}
and hence $\limsup_{n \to \infty} \frac1n \log \| T_n u_\infty \|
\le (-1 + 5 \varepsilon) \gamma$. Since this holds for every
$\varepsilon > 0$, we get \eqref{uinfupp}.
\end{proof}

\subsection{F\"urstenberg's Theorem}

In this subsection, let $\nu$ denote a probability measure on $\mathrm{SL}(2,\R)$ which
satisfies
\begin{equation}\label{logintegrable}
\int \log \| M \| \, d\nu(M) < \infty.
\end{equation}
Let us consider i.i.d.\ matrices $T_1,T_2,\ldots$, each
distributed according to $\nu$. Write $M_n = T_n \cdots T_1$. We
are interested in the Lyapunov exponent $\gamma \ge 0$, given by
$$
\gamma = \lim_{n \to \infty} \frac1n \log \| M_n \| , \quad
\nu^{\Z_+}-\text{a.s.}
$$
We are interested in conditions that ensure $\gamma > 0$. To
motivate the result below, let us give some examples with $\gamma
= 0$:
\begin{itemize}

\item If $\nu$ is supported in $\mathrm{SO}(2,\R)$, then $\gamma
= 0$.

\item If
$$
\nu \left\{ \left( \begin{array}{cc} 2 & 0 \\ 0 & 1/2 \end{array}
\right) \right\} = \frac12 \quad \text{ and } \quad \nu \left\{
\left( \begin{array}{cc} 1/2 & 0 \\ 0 & 2
\end{array} \right) \right\} = \frac12,
$$
then $\gamma = 0$: We have that
$$
M_n = \left( \begin{array}{cc} m_n & 0 \\ 0 & m_n^{-1} \end{array}
\right),
$$
where $\log m_n = a_1 + \cdots + a_n$ and $\{a_j\}$ are i.i.d.\
random variables taking values $\pm \log 2$, each with probability
$1/2$. Thus, $\log \| M_n \| = | a_1 + \cdots + a_n |$ and the
strong law of large numbers gives $\frac1n \log \| M_n \| \to 0$
almost surely.

\item If $p \in (0,1)$ and
$$
\nu \left\{ \left( \begin{array}{cc} 2 & 0 \\ 0 & 1/2 \end{array}
\right) \right\} = p \quad \text{ and } \quad \nu \left\{ \left(
\begin{array}{rc} 0 & 1 \\ -1 & 0
\end{array} \right) \right\} = 1-p,
$$
then $\gamma = 0$.

\end{itemize}

Furstenberg's Theorem shows that this list is essentially
exhaustive in the sense that the two mechanisms above, no growth
of norms or a finite (cardinality $=2$) invariant set of
directions, are the only ones that can preclude a positive
Lyapunov exponent.

Call two non-zero vectors $v_1,v_2$ in $\R^2$ equivalent if $v_2 =
\lambda v_1$ for some $\lambda \in \R$. The set of equivalence
classes is denoted by $\PP^1$. Since every $M \in
\mathrm{SL}(2,\R)$ is invertible, it induces a mapping from
$\PP^1$ to $\PP^1$ in the obvious way.

Let $\mathcal{M}(\PP^1)$ denote the set of probability measures
$m$ on $\PP^1$. Given $M \in \mathrm{SL}(2,\R)$ and $m \in
\mathcal{M}(\PP^1)$, we define $Mm \in\mathcal{M}(\PP^1)$ by
$$
\int f(v) \, d(Mm)(v) = \int f(Mv) \, dm(v).
$$
Moreover, we define the convolution $\nu \ast m \in
\mathcal{M}(\PP^1)$ by
$$
\int f(v) \, d(\nu \ast m)(v) = \iint f(Mv) \, d\nu(M) \, dm(v).
$$
If $\nu \ast m = m$, then $m$ is called $\nu$-\textit{invariant}.
By a Krylov-Bogoliubov argument, $\nu$-invariant measures always
exist.

%
We now state the main result of this section:

\begin{theorem}\label{furst}
Let $\nu$ be a probability measure on $\mathrm{SL}(2,\R)$ which
satisfies \eqref{logintegrable}. Denote by $G_\nu$ the smallest
closed subgroup of $\mathrm{SL}(2,\R)$ which contains
$\mathrm{supp} \, \nu$.

Assume
\begin{itemize}
\item[(i)] $G_\nu$ is not compact.
\end{itemize}
and one of the following conditions:
\begin{itemize}

\item[(ii)] There is no finite non-empty set $L \subseteq \PP^1$
such that $M(L) = L$ for all $M \in G_\nu$.

\item[(ii')] There is no set $L \subseteq \PP^1$ of cardinality
$1$ or $2$ such that $M(L) = L$ for all $M \in G_\nu$.


\end{itemize}
Then, $\gamma > 0$.
\end{theorem}

\noindent\textit{Remarks.} (a) We will show that (i)+(ii) implies $\gamma > 0$.\\
(b) To so see that (ii) can be replaced by (ii'), we remark that
(i)+(ii') implies~(ii). To prove this, it is enough to show that
if (i) holds and $v_1, \ldots , v_k$ are distinct elements of
$\PP^1$ with
$$
M ( \{ v_1, \ldots , v_k \} ) = \{ v_1, \ldots , v_k \} \quad
\text{for every } M \in G_\nu,
$$
then $k \le 2$. Each $M \in G_\nu$ induces a permutation $\pi(M)$
of $\{ v_1, \ldots , v_k \}$, and $\pi : G_\nu \to \mathcal{S}_k$
is a group homomorphism. The kernel $H$ of $\pi$ is a closed
normal subgroup of $G_\nu$, and $G_\nu / H$ is finite. By (i), $H$
is not finite. If $k \ge 3$, consider representatives of
$v_1,v_2,v_3$, which we denote by the same symbols. Write
$$
v_3 = \alpha v_1 + \beta v_2, \quad \alpha,\beta \not= 0.
$$
If $M \in H$, there are non-zero $\lambda_i$ such that $Mv_i =
\lambda_i v_i$, $i = 1,2,3$. This yields
\begin{align*}
\lambda_3 \alpha v_1 + \lambda_3 \beta v_2 & = \lambda_3 v_3 \\ &
= M v_3 \\& = \alpha M v_1 + \beta M v_2 \\ & = \alpha \lambda_1
v_1 + \beta \lambda_2 v_2.
\end{align*}
It follows that $\lambda_1 = \lambda_2 = \lambda_3$ and hence $M =
\lambda_1 \mathrm{Id}$. Since $M \in \mathrm{SL}(2,\R)$, this
shows $H \subseteq \{ \pm \mathrm{Id}
\}$ and $H$ is finite; contradiction.\\
(c) Note that the assumptions are monotonic in the support of the measure in the sense that if Theorem~\ref{furst} applies to $\nu$,
then it applies to any measure whose support contains the support of $\nu$.\\
(d) Theorem~\ref{furst} is a special case of a far more general result from \cite{fursten}.

\begin{lemma}\label{lemnonat}
If $\nu$ satisfies the assumption {\rm (ii)} of
Theorem~\ref{furst}, then every $\nu$-invariant measure is
non-atomic.
\end{lemma}

\begin{proof}
Assume that $m$ is $\nu$-invariant with
$$
w = \max \{ m(\{v\}) : v \in \PP^1 \} > 0.
$$
Let
$$
L = \{ v : m(\{v\}) = w \},
$$
which is a finite non-empty subset of $\PP^1$. It follows that,
for $v_0 \in L$,
\begin{align*}
w & = m(\{ v_0 \}) \\
& = (\nu \ast m) (\{ v_0 \}) \\
& = \iint \chi_{\{ v_0 \}}(Mv) \, d\nu(M) \, dm(v) \\
& = \iint \chi_{\{ M^{-1} v_0 \}}(v) \, dm(v) \, d\nu(M) \\
& = \int m \left( \{ M^{-1} v_0 \} \right) \, d\nu(M),
\end{align*}
where we used $\nu$-invariance in the second step.

On the other hand, $m \left( \{ M^{-1} v_0 \} \right) \le w$ for
all $M$ by the definition of $w$. Since the integral is equal to
$w$, we see that $m \left( \{ M^{-1} v_0 \} \right) = w$ for
$\nu$-almost every $M$. In other words, $\{ M^{-1} v_0 \} \in L$
for $\nu$-almost every $M$. It follows that $M^{-1}(L) \subseteq
L$ for $\nu$-almost every $M$, and hence, by finiteness of $L$,
for all $M$. This shows that $M(L) = L$ for all $M$, and so (ii)
fails.
\end{proof}

From now on, we assume that $\nu$ satisfies (i) and (ii), and $m
\in \mathcal{M}(\PP^1)$ is $\nu$-invariant (and hence non-atomic).
Our first goal is to express the Lyapunov exponent in terms of
these two measures.

\begin{lemma}
We have that
$$
\gamma = \iint \log \frac{\|Mv\|}{\|v\|} \, d\nu(M) \, dm(v).
$$
\end{lemma}

\begin{proof}
This follows quickly from Birkhoff and Osceledec-Ruelle: The shift
$\sigma$ on $\mathrm{SL}(2,\R)^{\Z_+}$, the space of sequences
$(T_1(\omega),T_2(\omega),\ldots)$, has the ergodic measure
$\nu^{\Z_+}$. The skew-product
$$
\tilde \sigma : \mathrm{SL}(2,\R)^{\Z_+} \times \PP^1 \to
\mathrm{SL}(2,\R)^{\Z_+} \times \PP^1 , \quad (\omega,v) \mapsto
(\sigma \omega, T_1(\omega)v)
$$
leaves invariant the measure $\nu^{\Z_+} \times m$. Consider the
function
\begin{equation}\label{furstprooffdef}
f(\omega,v) = \log \frac{\|T_1(\omega) v\|}{\|v\|}.
\end{equation}
Then,
$$
\frac1n \sum_{m=1}^n f(\tilde \sigma^m(\omega,v)) = \frac1n \log
\frac{\|T_n(\omega) \cdots T_1(\omega) v\|}{\|v\|} = \frac1n \log
\frac{\|M_n(\omega) v\|}{\|v\|},
$$
which, on the one hand, converges to $\gamma$ for almost every
$\omega$ and almost every $v$ by Osceledec-Ruelle, and, on the
other hand, Birkhoff's Theorem gives that it converges to $\iint f
\, d\nu^{\Z_+}\, dm$. Putting these two things together,
\begin{equation}\label{furstgammaident}
\gamma = \iint \log \frac{\|T_1(\omega) v\|}{\|v\|} \,
d\nu^{\Z_+}(\omega)\, dm(v) = \iint \log \frac{\|Mv\|}{\|v\|} \,
d\nu(M) \, dm(v),
\end{equation}
as claimed.
\end{proof}

Let us prove two auxiliary lemmas, which will be useful later in
the proof of Theorem~\ref{furst}.

\begin{lemma}\label{auxlem1}
If $m \in \mathcal{M}(\PP^1)$ is non-atomic and $M_n \not= 0$
converge to $M \not= 0$, then $M_n m \to M m$ weakly.
\end{lemma}

\begin{proof}
Since $M \not= 0$, there is at most one direction $v$ for which
$Mv$ is not defined (because the vectors in the direction of $v$
are in $\mathrm{Ker} \, M$), similarly for all the $M_n$. Thus,
for $v$ outside a countable set $C \subseteq \PP^1$, $M_n v$ and
$M v$ are defined and we have $M_n v \to M v$ as $n \to \infty$.
If $f \in C(\PP^1)$, we therefore get
$$
\int f(v) \, d(M_n m) (v) = \int f(M_n v) \, dm(v) \to \int f(M v)
\, dm(v) = \int f(v) \, d(M m) (v)
$$
by dominated convergence and $m(C) = 0$. That is, $M_n m \to M m $
weakly.
\end{proof}

\begin{lemma}\label{auxlem2}
If $m \in \mathcal{M}(\PP^1)$ is non-atomic, then
$$
H = \{ M \in \mathrm{SL}(2,\R) : M m = m \}
$$
is a compact subgroup of $\mathrm{SL}(2,\R)$.
\end{lemma}

\begin{proof}
It is clear that $H$ is a closed subgroup of $\mathrm{SL}(2,\R)$
so we only need to prove boundedness. Assume that there are $M_n
\in H$ such that $\| M_n \| \to \infty$. For a suitable
subsequence, $\|M_{n_k}\|^{-1} M_{n_k}$ converges to a matrix
$M_\infty \not= 0$, which obeys $M_\infty m = m$ by
Lemma~\ref{auxlem1} since $m$ is non-atomic. On the other hand,
$$
\det M_\infty  = \lim_{n \to \infty} \det
\frac{M_{n_k}}{\|M_{n_k}\|} = \lim_{n \to \infty}
\frac{1}{\|M_{n_k}\|^2} = 0,
$$
so $M_\infty$ is rank one and $m$ must be a Dirac measure because
of $M_\infty m = m$; which is a contradiction.
\end{proof}

We return to the proof of Theorem~\ref{furst} and show that the
measures $M_n(\omega) m \in \mathcal{M}(\PP^1)$ converge weakly to
a Dirac measure with a certain invariance property for almost
every $\omega \in \mathrm{SL}(2,\R)^{\Z_+}$.

\begin{lemma}\label{furstlemmlong}
For $\nu^{\Z_+}$-almost every $\omega$, we have the following:\\
{\rm (a)} There exists $m_\omega \in \mathcal{M}(\PP^1)$ such that
$M_n(\omega) m \to
m_\omega$ weakly.\\
{\rm (b)} For $\nu$-almost every $M \in \mathrm{SL}(2,\R)$,
$M_n(\omega) M m \to
m_\omega$ weakly.\\
{\rm (c)} There exists $v_\omega \in \PP^1$ such that $m_\omega =
\delta_{v_\omega}$.
\end{lemma}

\begin{proof}
(a) Fix $g \in C(\PP^1)$ and define $G : \mathrm{SL}(2,\R) \to \R$
by
$$
G(M) = \int g(Mv) \, dm(v).
$$
Let $\mathcal{F}_n$ be the $\sigma$-algebra of
$\mathrm{SL}(2,\R)^{\Z_+}$ formed by the cylinders of length $n$.
Then $M_n(\cdot)$ is $\mathcal{F}_n$-measurable. We want to show
that $G(M_n(\omega))$ converges for almost every $\omega$. We will
employ the Martingale Convergence Theorem; see \cite[(2.10]{Du}
for this result and \cite[Sections~4.1 and 4.2]{Du} for
background.

We have that
\begin{align*}
\E (G(M_{n+1}) | \mathcal{F}_n) & = \int G(M_n M) \, d\nu(M) \\
& = \iint g(M_n M v) \, d\nu(M) dm(v) \\
& = \int g(M_n v) \, dm(v) \\
& = G(M_n),
\end{align*}
where we used $\nu$-invariance of $m$ in the last step. This shows
that $\omega \mapsto G(M_n(\omega))$ is a martingale, and hence
the limit
$$
\Gamma_g(\omega) = \lim_{n \to \infty} G(M_n(\omega))
$$
exists for almost every $\omega$ by the Martingale Convergence
Theorem.

Now pick a countable dense subset $\{g_k\}$ of $C(\PP^1)$ and take
$\omega$ from the full measure subset where $\Gamma_{g_k}(\omega)$
exists for all $k$. Let $m_\omega$ be a weak accumulation point of
the sequence $\{ M_n(\omega) m \}$. Then
$$
\int g_k \, dm_\omega = \lim_{j \to \infty} \int g_k \,
d(M_{n_j}(\omega) m) = \lim_{j \to \infty} \int g_k \circ M_{n_j}
(\omega) dm = \Gamma_{g_k} (\omega).
$$
Since the limit is the same for every subsequence, we have in fact
that $M_n(\omega) m
\to m_\omega$ weakly, as desired.\\
(b) We have to show that for any $g \in C(\PP^1)$,
\begin{equation}\label{partbgoal}
\lim_{n \to \infty} \E (G(M_n M)) = \Gamma(g) = \lim_{n \to
\infty} \E (G(M_n)) , \quad \nu-\text{a.e. } M \in
\mathrm{SL}(2,\R),
\end{equation}
where $\E$ is integration over $\omega$. We will show
\begin{equation}\label{partbgoalact}
\lim_{n \to \infty} \E ([G(M_{n+1}) - G(M_n)]^2) = 0.
\end{equation}
Since
$$
\E ([G(M_{n+1}) - G(M_n)]^2) = \E \left( \left[ \iint ( g(M_n M v)
- g(M_n v) \, dm(v) \, d\nu(M) \right]^2 \right),
$$
we may deduce from \eqref{partbgoalact}, for almost every
$\omega$,
$$
\lim_{n \to \infty} \int G(M_n(\omega)M) - G(M_n(\omega)) \,
d\nu(M) = \lim_{n \to \infty} \iint g(M_n(\omega)M v) -
g(M_n(\omega)) \, dm(v) \, d\nu(M),
$$
which implies \eqref{partbgoal}.

Note that
$$
\E \left( [ G(M_{n+1}) - G(M_n) ]^2 \right) = \E \left(
G(M_{n+1})^2 \right) + \E \left( G(M_n)^2 \right) - 2 \E \left(
G(M_{n+1})G(M_n) \right)
$$
and
\begin{align*}
\E \left( G(M_{n+1})G(M_n) \right) & = \E \left( \int g(M_{n+1}v)
\, dm(v) \, \cdot \,
\int g(M_n v) \, dm(v) \right) \\
& = \E \left( \iint g(M_n M v) \, dm(v) \, d\nu(M) \, \cdot \,
\int g(M_n v) \, dm(v) \right) \\
& = \E \left( \left[ \int g(M_n v) \, dm(v) \right]^2 \right) \\
& = \E\left( G(M_n)^2 \right).
\end{align*}
Thus,
$$
\E \left( [ G(M_{n+1}) - G(M_n) ]^2 \right) = \E \left(
G(M_{n+1})^2 \right) - \E \left( G(M_n)^2 \right)
$$
and hence
\begin{align*}
\sum_{n=1}^N \E \left( [ G(M_{n+1}) - G(M_n) ]^2 \right) & =
\sum_{n=1}^N \E \left(
G(M_{n+1})^2 \right) - \E \left( G(M_n)^2 \right) \\
& = \E \left( G(M_{N+1})^2 \right) - \E \left( G(M_1)^2 \right) \\
& \le \|g\|_\infty.
\end{align*}
Thus, $\sum_{n=1}^\infty \E \left( [ G(M_{n+1}) - G(M_n) ]^2
\right) < \infty$ and
\eqref{partbgoalact} follows.\\
(c) Consider an $\omega$ for which (a) and (b) hold, that is,
$$
M_n(\omega) m \to m_\omega \; \text{ and } \; M_n(\omega) M m \to
m_\omega, \quad \nu-\text{a.e. }M \in \mathrm{SL}(2,\R).
$$
Let $M(\omega)$ be an accumulation point of the sequence $\{
\|M_n(\omega)\|^{-1} M_n(\omega) \}$. Since $\nu$ is non-atomic,
Lemma~\ref{auxlem1} implies
$$
M(\omega) m = M(\omega) M m = m_\omega, \quad \nu-\text{a.e. }M
\in \mathrm{SL}(2,\R).
$$
If $M(\omega)$ is invertible, it follows that
$$
m = M m, \quad \nu-\text{a.e. }M \in \mathrm{SL}(2,\R).
$$
But
$$
H = \{ M \in \mathrm{SL}(2,\R) : m = M m  \}
$$
is compact by Lemma~\ref{auxlem2}, which contradicts assumption
(i) from Theorem~\ref{furst}. It follows that $M(\omega)$ is not
invertible, that is, its range is one-dimensional. But this
implies the assertion since $M(\omega) m = m_\omega$.
\end{proof}

The next step is to show that convergence to a Dirac measure
implies norm growth.

\begin{lemma}\label{furstlemmshort}
Let $m \in \mathcal{M}(\PP^1)$ be non-atomic and let $\{M_n\} \in
\mathrm{SL}(2,\R)^{\Z_+}$ with $M_n m \to \delta_v$ weakly for
some $v \in \PP^1$. Then,
$$
\lim_{n \to \infty} \|M_n\| = \lim_{n \to \infty} \| M_n^* \| =
\infty.
$$
Moreover,
$$
\frac{\|M_n^* w\|}{\|M_n^*\|} \to \left| \langle w ,
\tfrac{v}{\|v\|} \rangle \right|
$$
for every $w \in \R^2$.
\end{lemma}

\begin{proof}
Assume first that there is a matrix $M$ such that
$$
\frac{M_n}{\|M_n\|} \to M.
$$
By Lemma~\ref{auxlem1} we have that
$$
\frac{M_n}{\|M_n\|}m \to Mm
$$
weakly. Thus, by assumption, $Mm = \delta_v$. This shows that
$\det M = 0$ since $m$ is non-atomic (otherwise, $m = M^{-1}
\delta_v = \delta_{M^{-1}v}$). Therefore,
$$
0 = \det M = \lim_{n \to \infty} \det \frac{M_n}{\|M_n\|} =
\lim_{n \to \infty} \frac{1}{\|M_n\|^2},
$$
which yields the first assertion.

From $Mm = \delta_v$ we also infer that the range of $M$ is the
line in $v$-direction. Denote a unit vector in this direction by
the same symbol, $v$. If $e_1 = (1,0)^T$ and $e_2 = (0,1)^T$, we
can therefore write
$$
M e_1 = \pm  \|Me_1\| v, \quad M e_2 = \pm  \|Me_2\| v.
$$
Consider a vector $w \in \R^2$. We have that
\begin{align*}
\| M^* w \|^2 & = |\langle M^* w , e_1 \rangle|^2 + |\langle M^* w , e_2 \rangle|^2 \\
 & = |\langle w , M e_1 \rangle|^2 + |\langle w , M e_2 \rangle|^2 \\
 & = \left( \|Me_1\|^2 + \|Me_2\|^2 \right) |\langle w , v \rangle|^2
\end{align*}
This shows $\|Me_1\|^2 + \|Me_2\|^2 = 1$ (set $w = v$ and use
$\|M^*\|=1$) and hence $\| M^* w \| = |\langle w , v \rangle|$ for
all $w \in \R^2$, from which the second assertion follows.

If $M_n/\|M_n\|$ does not converge, we can perform the steps above
for each convergent subsequence and get the same limits for all
such subsequences. This gives the claims for the entire sequence.
\end{proof}

Our final goal is to prove that the norm growth must actually be exponentially fast. We first prove the following lemma, which will be helpful in this regard.

\begin{lemma}\label{furstprooflemmmid}
Let $T$ be a measure preserving transformation on a probability space $(\Omega,d\mu)$. If $f \in L^1(\Omega,d\mu)$ is such that
$$
\lim_{n \to \infty} \sum_{m=0}^{n-1} f(T^m \omega) = \infty
$$
for $\mu$-almost every $\omega$, then $\E(f) > 0$.
\end{lemma}

\begin{proof}
The Birkhoff Theorem shows that we may define a function $\tilde f$ for $\mu$-almost every $\omega$ by
$$
\tilde f(\omega) = \lim_{n \to \infty} \frac1n \sum_{m=0}^{n-1} f(T^m \omega).
$$
By assumption on $f$, $\tilde f \ge 0$. Assume that $\E(f) = 0$. Then, again by Birkhoff, $\tilde f(\omega) = 0$ for $\omega \in \Omega_0$, where $\Omega_0$ is $T$-invariant and $\mu(\Omega_0) = 1$.

Write
$$
s_n(\omega) = \sum_{m=0}^{n-1} f(T^m \omega).
$$
For $\varepsilon > 0$, let
$$
A_\varepsilon = \{ \omega \in \Omega_0 : s_n(\omega) \ge \varepsilon \text{ for every } n \ge 1 \}
$$
and
$$
B_\varepsilon = \bigcup_{k \ge 0} T^{-k} \left( A_\varepsilon \right).
$$
For $\omega \in B_\varepsilon$, denote by $k(\omega) \ge 0$ the
smallest integer with $T^{k(\omega)}\omega \in A_\varepsilon$.
Then, for $n \ge k(\omega)$,
$$
s_n(\omega) = s_{k(\omega)}(\omega) + s_{n -
k(\omega)}(T^{k(\omega)}\omega) \ge s_{k(\omega)}(\omega) +
\sum_{m = k(\omega)}^{n-1} \varepsilon \chi_{A_\varepsilon}(T^m
\omega).
$$
Dividing by $n$ and taking $n$ to infinity, we get
$$
0 = \tilde f(\omega) \ge \varepsilon \tilde \chi_{A_\varepsilon} (
\omega ),
$$
where $\tilde \chi_{A_\varepsilon}$ is defined through Birkhoff as
before. Then,
$$
\mu (A_\varepsilon) = \E ( \tilde \chi_{A_\varepsilon}) = \E (
\tilde \chi_{A_\varepsilon} \chi_{B_\varepsilon}) = 0.
$$
Since $T$ is measure-preserving, we get $\mu (B_\varepsilon) = 0$,
which also implies
$$
\mu \left( \bigcup_{\varepsilon > 0} B_\varepsilon \right) = 0.
$$
This is a contradiction since $s_n(\omega) \to \infty$ implies
$\omega \in \bigcup_{\varepsilon > 0} B_\varepsilon$.
\end{proof}

\begin{proof}[Proof of Theorem~\ref{furst}.]
Consider the function $f : \mathrm{SL}(2,\R)^{\Z_+} \times \PP^1 \to \mathrm{SL}(2,\R)^{\Z_+} \times \PP^1$, $f(\omega,v) = \log \frac{\|T_1(\omega) v\|}{\|v\|}$; compare \eqref{furstprooffdef}. We saw above that $\gamma = \E (f)$; see \eqref{furstgammaident}. Therefore, our goal is to show that $\E(f) > 0$.

Note that $\nu^*$ satisfies the assumptions (i) and (ii) if $\nu$ does (since $Mv = w \Rightarrow M^*(w^\perp) = v^\perp$). Thus, by Lemmas~\ref{furstlemmlong} and \ref{furstlemmshort}, we have that
$$
\sum_{m=0}^{n-1} f(\tilde \sigma^m(\omega,v)) = \log \frac{\|M_n^*(\omega) v\|}{\|v\|} \to \infty
$$
for almost every $\omega$ and every $v \in \PP^1 \setminus \{v_\omega^\perp\}$. In particular, this divergence holds $\nu^{\Z_+} \times m$-almost surely. Thus, it follows from Lemma~\ref{furstprooflemmmid} that $\E(f) > 0$, which concludes the proof.
\end{proof}

\subsection{Application of F\"urstenberg's Theorem in the Discrete Case}\label{ss.furd}

Let us now apply F\"urstenberg's Theorem to the Anderson model. Recall that it is given by a probability measure $\tilde \nu$ on $\R$ with $\mathrm{supp} \, \tilde \nu$ compact and of cardinality $\ge 2$. For every $E \in \R$, the measure $\tilde \nu$ induces a measure $\nu$ on $\mathrm{SL}(2,\R)$ via
$$
v \mapsto \left( \begin{array}{cr} E - v & -1 \\ 1 & 0 \end{array}
\right).
$$
The definitions are such that the Lyapunov exponent associated with this $\nu$ at the beginning of this subsection is equal to $\gamma(E)$ defined earlier.

\begin{theorem}
In the Anderson model, we have $\gamma(E) > 0$ for every $E \in \R$.
\end{theorem}

\begin{proof}
Let us check that F\"urstenberg's Theorem applies. Fix any $E \in \R$. Since $\mathrm{supp} \, \tilde \nu$ has cardinality $\ge 2$, $\mathrm{supp} \, \nu$ has cardinality $\ge 2$, and hence $G_\nu$ contains at least two distinct elements of the form
$$
M_x = \left( \begin{array}{cr} x & -1 \\ 1 & 0 \end{array} \right),
$$
for example, $M_a$ and $M_b$ with $a \not= b$. Note that
$$
M^{(1)} = M_a M_b^{-1} = \left( \begin{array}{cc} 1 & a-b \\ 0 & 1 \end{array} \right) \in G_\nu.
$$
Taking powers of the matrix $M^{(1)}$, we see that $G_\nu$ is not compact.

Consider the equivalence class of $e_1 = (1,0)^T$ in $\PP^1$. Then $M^{(1)} e_1 = e_1$ and for every $v \in \PP^1$, $(M^{(1)})^n v$ converges to $e_1$. Thus, if there is a finite invariant set of directions $L$, it must be equal to $\{e_1\}$. However,
$$
M^{(2)} = M_a^{-1} M_b = \left( \begin{array}{cc} 1 & 0 \\ a-b & 1
\end{array} \right) \in G_\nu
$$
and $M^{(2)} e_1 \not= e_1$; contradiction. Thus, the conditions (i) and (ii) of Theorem~\ref{furst} hold and, consequently, $\gamma(E) > 0$.
\end{proof}

\subsection{Application of F\"urstenberg's Theorem in the Continuum Case}\label{ss.furc}

See \cite{dss}.

\section{Proving Localization Given Positive Lyapunov Exponents}\label{s.locposle}

\subsection{Spectral Averaging}\label{ss.spav}

Suppose that $A \in B(\ell^2(\Z))$ is self-adjoint and $\phi \in \ell^2(\Z) \setminus \{ 0 \}$. With the associated spectral measure $\mu$, we therefore have
$$
F(z) := \langle \phi, (A - z)^{-1} \phi \rangle = \int \frac{d\mu(E)}{E-z}, \quad z \in \C \setminus \R.
$$
For $\lambda \in \R$, we consider the operator
$$
A_\lambda = A + \lambda \langle \phi, \cdot \rangle \phi,
$$
which is a rank one perturbation of $A$. Define $F_\lambda$ and $\mu_\lambda$ by
\begin{equation}\label{f.flambdadef}
F_\lambda(z) = \langle \phi, (A_\lambda - z)^{-1} \phi \rangle = \int \frac{d\mu_\lambda(E)}{E-z}, \quad z \in \C \setminus \R.
\end{equation}

\begin{lemma}
We have
\begin{equation}\label{f.aronkrein}
F_\lambda(z) = \frac{F(z)}{1 + \lambda F(z)}.
\end{equation}
\end{lemma}

\begin{proof}
Using the resolvent formula in the second step, we see that
\begin{align*}
(A_\lambda - z)^{-1} \phi - (A - z)^{-1} \phi & = \left[(A_\lambda - z)^{-1} - (A - z)^{-1}\right] \phi \\
& = (A - z)^{-1} \left( A - A_\lambda \right) (A_\lambda - z)^{-1} \phi\\
& = (A - z)^{-1} \left( - \lambda \langle \phi, \cdot \rangle \phi \right) (A_\lambda - z)^{-1} \phi\\
& = (A - z)^{-1} \left( - \lambda \langle \phi, (A_\lambda - z)^{-1} \phi \rangle \phi \right) \\
& = -\lambda \langle \phi , (A_\lambda - z)^{-1} \phi \rangle (A - z)^{-1} \phi
\end{align*}
and hence, by taking the the inner product with $\phi$ on both sides,
$$
F_\lambda(z) - F(z) = - \lambda F_\lambda(z) F(z).
$$
Solving this for $F_\lambda(z)$, we obtain \eqref{f.aronkrein}.
\end{proof}

\begin{theorem}\label{t.specaver}
We have
$$
\int \left[ d\mu_\lambda(E) \right] \, d\lambda = dE
$$
in the sense that if $f \in L^1(\R,dE)$, then $f \in L^1(\R,d\mu_\lambda)$ for Lebesgue almost every $\lambda$, $\int f(E) \, d\mu_\lambda(E) \in L^1(\R,d\lambda)$, and
\begin{equation}\label{f.specavertp}
\int \left( \int f(E) \, d\mu_\lambda(E) \right) \, d\lambda = \int f(E) \, dE.
\end{equation}
\end{theorem}

\begin{proof}
Denote for $E,\lambda \in \R$ and $z \in \C \setminus \R$,
$$
f_z(E) = \frac{1}{E-z} - \frac{1}{E+i}
$$
and
$$
h_z(\lambda) = \frac{1}{\lambda + F(z)^{-1}} - \frac{1}{\lambda + F(-i)^{-1}}.
$$
By closing the contour in the upper half-plane, we see that
$$
\int f_z(E) \, dE = \begin{cases} 2\pi i & \Im z > 0, \\ 0 & \Im z < 0. \end{cases}
$$
By \eqref{f.flambdadef} and \eqref{f.aronkrein} we have
\begin{align*}
\int f_z(E) \, d\mu_\lambda(E) & = F_\lambda(z) - F_\lambda(-i) \\
& = \frac{1}{\lambda + F(z)^{-1}} - \frac{1}{\lambda + F(-i)^{-1}} \\
& = h_z(\lambda).
\end{align*}
Observe that if $\pm \Im z > 0$, then $\pm \Im F(z) > 0$ and hence $\pm \Im F(z)^{-1} < 0$. Thus, $h_z(\lambda)$ has either two poles in the lower half-plane (if $\Im z < 0$) or one in each half-plane (if $\Im z > 0$). Thus, the same contour integral calculation as above (with $\lambda$ instead of $E$) shows
$$
\int \left( \int f_z(E) \, d\mu_\lambda(E) \right) \, d\lambda = \int h_z(\lambda) \, d\lambda =
\begin{cases} 2\pi i & \Im z > 0, \\ 0 & \Im z < 0, \end{cases}
$$
which proves \eqref{f.specavertp} for $f = f_z$. The general case then follows from the fact the the finite linear combinations of the $f_z$'s are dense (which can be shown with the help of Stone-Weierstrass).
\end{proof}

\subsection{Localization via Spectral Averaging}

\begin{theorem}\label{kotlocwl}
For the Anderson model on the line with a single-site distribution that has an absolutely continuous component, we have that $H_\omega$ is spectrally localized for $\mu$-almost every $\omega \in \Omega$.
\end{theorem}

\begin{proof}
By F\"urstenberg's Theorem, the Osceledec-Ruelle Theorem, and Fubini, we have that
\begin{equation}\label{f.noweight}
\mathrm{Leb} (\R \setminus \mathcal{E}_\omega) = 0
\end{equation}
for almost every $\omega$, where
$$
\mathcal{E}_\omega = \{ E \in \R : \gamma(E) > 0 , \, \exists \text{ solutions } u_\pm \text{ with } |u_\pm(n)| \sim e^{-\gamma(E) |n|} \text{ as } n \to \pm \infty \}.
$$
Note that the sets $\mathcal{E}_\omega$ are invariant with respect to a modification of $V_\omega$ on a finite set! We will perform such a modification, within the family $\{V_\omega\}$, on the set $\{0,1\}$ because the pair $\{ \delta_0, \delta_1 \}$ is cyclic for each operator $H_\omega$.

Denote the set of $\omega$'s for which \eqref{f.noweight} holds by $\Omega_0$. We know that
\begin{equation}\label{f.omega0fm}
\mu(\Omega_0) = 1.
\end{equation}
For $\omega \in \Omega_0$, consider the operators
$$
H_{\omega,\lambda_0,\lambda_1} = H_\omega + \lambda_0 \langle \delta_0 , \cdot \rangle \delta_0 + \lambda_1 \langle \delta_1 , \cdot \rangle \delta_1,
$$
where $\lambda_0, \lambda_1 \in \R$. For every fixed $\lambda_0$, it follows from Theorem~\ref{t.specaver} and \eqref{f.noweight} that the spectral measure of the pair $(H_{\omega,\lambda_0,\lambda_1} , \delta_1)$ gives zero weight to the set $\R \setminus \mathcal{E}_\omega$ for Lebesgue almost every $\lambda_1 \in \R$. Similarly, for every fixed $\lambda_1$, the spectral measure of the pair $(H_{\omega,\lambda_0,\lambda_1} , \delta_0)$ gives zero weight to the set $\R \setminus \mathcal{E}_\omega$ for Lebesgue almost every $\lambda_0 \in \R$. As a consequence, we find that for Lebesgue almost every $(\lambda_0,\lambda_1) \in \R^2$, the universal spectral measure of $H_{\omega,\lambda_0,\lambda_1}$ (the sum of the spectral measures of $\delta_0$ and $\delta_1$) gives zero weight to the set $\R \setminus \mathcal{E}_\omega$. Write $G_\omega$ for this set of ``good'' pairs $(\lambda_0,\lambda_1)$, so that
\begin{equation}\label{f.gomegafm}
\mathrm{Leb} (\R^2 \setminus G_\omega) = 0.
\end{equation}
Let
$$
\Omega_1 = \{ \omega + \lambda_0 \delta_0 + \lambda_1 \delta_1 : \omega \in \Omega_0 , \; (\lambda_0,\lambda_1) \in G_\omega \}.
$$
Since $\nu_\mathrm{ac} \not= 0$, it follows that from \eqref{f.omega0fm} and \eqref{f.gomegafm} that
$$
\mu(\Omega_1) > 0.
$$
Thus, by assumption on $\nu$, with positive $\nu \times \nu$ probability, it follows from \eqref{f.noweight} that the whole-line spectral measure (corresponding to the sum of the $\delta_0$ and $\delta_1$ spectral measures) assigns no weight to $\R \setminus \mathcal{E}_\omega$ and hence, with positive $\mu$ probability, the operator $H_\omega$ is spectrally localized by subordinacy theory.

Since localization is a shift-invariant event, the operator $H_\omega$ must in fact be spectrally localized for $\mu$-almost every $\omega$.
\end{proof}

\subsection{The Case of Singular Distributions}\label{ss.locsingdis}

See \cite{ckm, dss}.

\section{The Kunz-Souillard Approach to Localization}\label{s.ks}

In this section we present the Kunz-Souillard approach to localization in one dimension. This approach targets dynamical localization directly and has the additional advantage that one can throw in a background potential essentially for free. As a consequence one may show that a given deterministic potential may be perturbed by an arbitrarily small random potential so that the resulting Schr\"odinger operator is dynamically localized. In this sense, randomness dominates any other kind of disorder in one dimension and imposes its fingerprint --- localization.

The original Kunz-Souillard work was carried out in the discrete case; see \cite{ks} for the original paper and \cite{cfks, desiso, dss2} for related material. There is work in the continuum inspired by Kunz-Souillard due to Royer \cite{r}. A full continuum analogue of the discrete Kunz-Souillard method was recently worked out by Stolz and the present author \cite{ds}.

The Kunz-Souillard method applies to single-site distributions $\nu$ that are purely absolutely continuous. This is in some sense the price one has to pay for a very direct localization proof with a very strong conclusion. It is an interesting open problem to extend the scope of this approach to single-site distributions with a non-trivial singular component.

We present the Kunz-Souillard method in the discrete case in Subsection~\ref{ss.ksd} and then discuss the continuum case in Subsection~\ref{ss.ksc}

\subsection{The Discrete Case}\label{ss.ksd}

We assume that the probability measure $\nu$ on $\R$ is purely absolutely continuous. More precisely, we assume that there is a bounded and compactly supported density $r$ such that $d\nu(E) = r(E) \, dE$. Then, we define $\Omega$, $V_\omega$, $H_\omega$ as before; compare \eqref{e.operd} and Definition~\ref{d.potd}. Our main goal in this subsection is to prove the following result:

\begin{theorem}\label{t.ksd}
Under the assumptions from the previous paragraph, there are $C,\gamma \in (0,\infty)$ such that for $m,n \in \Z$, we have
\begin{equation}\label{e.ks1}
\int_\Omega \left( \sup_{t \in \R} \left| \left\langle \delta_m , e^{-itH_\omega} \delta_n \right\rangle \right| \right) \, d\mu(\omega) \le C e^{-\gamma |m-n|}.
\end{equation}
\end{theorem}

For the proof of this result, we will mostly follow the excellent exposition from \cite[Section~9.5]{cfks} (interesting extensions may be found in \cite{desiso, simon}). For $m,n \in \Z$, denote
\begin{equation}\label{e.amndef}
a(m,n) = \int_\Omega \left( \sup_{t \in \R} \left| \left\langle \delta_m , e^{-itH_\omega} \delta_n \right\rangle \right| \right) \, d\mu(\omega).
\end{equation}
Given $L \in \Z_+$, denote by $H_\omega^{(L)}$ the restriction of $H_\omega$ to $\ell^2(-L,\ldots,L)$ and let
\begin{equation}\label{e.ks2}
a_L(m,n) = \int_\Omega \left( \sup_{t \in \R} \left| \left\langle \delta_m , e^{-itH_\omega^{(L)}} \delta_n \right\rangle \right| \right) \, d\mu(\omega).
\end{equation}
Note, however, that $H_\omega^{(L)}$ depends only on the entries $\omega_{-L},\ldots,\omega_L$ of $\omega$ and hence the expectation in \eqref{e.ks2} is given by a $(2L+1)$-fold integral over $\R$.

\begin{lemma}\label{l.ks1}
For $m,n \in \Z$, we have
$$
a(m,n) \le \limsup_{L \to \infty} a_L(m,n).
$$
\end{lemma}

\begin{proof}
Regard $H_\omega^{(L)}$ as an operator on $\ell^2(\Z)$ (i.e., set $\langle \delta_j ,  H_\omega^{(L)} \delta_k \rangle = 0$ if $|j| > L$ or $|k| > L$). Clearly, $H_\omega^{(L)}$ converges strongly to $H_\omega$ as $L \to \infty$. Thus, $e^{-itH_\omega^{(L)}}$ converges strongly to $e^{-itH_\omega}$ as $L \to \infty$. Consequently, the assertion follows by Fatou's Lemma.
\end{proof}

The operator $H_\omega^{(L)}$ is self-adjoint on $\ell^2(-L,\ldots,L)$; we denote its eigenvalues by $\{ E_\omega^{L,k} \}$ and the corresponding normalized eigenvectors by $\{ \varphi_\omega^{L,k} \}$. Let
$$
\varrho_L(m,n) =  \int_\Omega \left( \sum_k \left| \left\langle \delta_m , \varphi_\omega^{L,k} \right\rangle \right| \, \left| \left\langle \delta_n , \varphi_\omega^{L,k} \right\rangle \right| \right) \, d\mu(\omega).
$$
Note that the sum consists of $2L + 1$ terms.

\begin{lemma}\label{l.ks2}
For $L \in \Z_+$ and $m,n \in \Z$, we have
$$
a_L(m,n) \le \varrho_L(m,n).
$$
\end{lemma}

\begin{proof}
We have
\begin{align*}
a_L(m,n) & = \int_\Omega \left( \sup_{t \in \R} \left| \left\langle \delta_m , e^{-itH_\omega^{(L)}} \delta_n \right\rangle \right| \right) \, d\mu(\omega) \\
& = \int_\Omega \left( \sup_{t \in \R} \Big| \Big\langle \delta_m , e^{-itH_\omega^{(L)}} \sum_k \left\langle \varphi_\omega^{L,k} , \delta_n \right\rangle \varphi_\omega^{L,k} \Big\rangle \Big| \right) \, d\mu(\omega) \\
& \le \int_\Omega \left( \sum_k \left| \left\langle \delta_m , \varphi_\omega^{L,k} \right\rangle \right| \, \left| \left\langle \delta_n , \varphi_\omega^{L,k} \right\rangle \right| \right) \, d\mu(\omega) \\
& = \varrho_L(m,n),
\end{align*}
as claimed.
\end{proof}

Given Lemmas~\ref{l.ks1} and \ref{l.ks2}, our goal is to prove $\varrho_L(m,n) \le C e^{-\gamma |m-n|}$ with $L$-independent constants $C,\gamma \in (0,\infty)$. It suffices to prove this estimate for $m \ge 0$ and $n = 0$. This will be done in three steps:
\begin{itemize}

\item[(a)] Represent $\varrho_L(m,0)$ with the help of a product of integral operators $T_0$ and $T_1$  (Lemma~\ref{l.ks3}).

\item[(b)] Prove estimates for $T_0$ and $T_1$  (Lemma~\ref{l.ks4}).

\item[(c)] Derive the desired exponential bound for $\varrho_L(m,0)$.

\end{itemize}

Recall that $r$ denotes the density of the absolutely continuous single-site distribution.

\begin{definition}
Denote $M = \max \{ |E| : E \in \mathrm{supp} \, r \}$ and $\Sigma_0 = [- 2 -M , 2 + M]$.
\end{definition}

The spectra of all operators $H_\omega$ and $H_\omega^{(L)}$ are contained in $\Sigma_0$.

\begin{definition}
For $E \in \R$, define the operators $U , T_0 , T_1$ on $L^p(\R)$ by
\begin{align*}
Uf(x) & = |x|^{-1} f(x^{-1}) \\
T_0f(x) & = \int r(E - x - y^{-1}) f(y) \, dy \\
T_1f(x) & = \int r(E - x - y^{-1}) |y|^{-1} f(y) \, dy
\end{align*}
\end{definition}

\begin{lemma}\label{l.ks3}
With $\phi(x) = r(E - x)$, we have
$$
\varrho_L(m,0) = \int_{\Sigma_0} \left\langle T_1^{m-1} T_0^{L-m} \phi , U T_0^L \phi \right\rangle_{L^2(\R)} \, dE.
$$
\end{lemma}

\begin{proof}
As pointed out earlier, $\varrho_L(m,0)$ can be expressed by a $(2L+1)$-fold integral:
\begin{equation}\label{e.ks8}
\varrho_L(m,0) = \int \cdots \int \left( \sum_k \left| \left\langle \delta_m , \varphi_{\tilde V}^{L,k} \right\rangle \right| \, \left| \left\langle \delta_0 , \varphi_{\tilde V}^{L,k} \right\rangle \right| \right) \prod_{n = -L}^L r(V_n) \, dV_{-L} \cdots dV_L,
\end{equation}
where $\tilde V = (V_{-L},\ldots,V_L)$ and we denote the eigenvalues (listed in increasing order) and corresponding normalized eigenvectors of
$$
\begin{pmatrix} V_{-L} & 1 & &&& \\ 1 & V_{-L+1} & 1 &&& \\ & \ddots & \ddots & \ddots &&  \\ && \ddots & \ddots & \ddots & \\ &&& 1 & V_{L-1} & 1 \\ &&&& 1 & V_L \end{pmatrix}
$$
by $\{ E_{\tilde V}^{L,k} \}_{-L \le k \le L}$ and $\{ \varphi_{\tilde V}^{L,k} \}_{-L \le k \le L}$, respectively.

Thus, if $E$ is $E_{\tilde V}^{L,k}$ and $u$ is $\varphi_{\tilde V}^{L,k}$, then we have
\begin{equation}\label{e.ks5}
u(n+1) + u(n-1) + V_n u(n) = E u(n)
\end{equation}
for $-L \le n \le L$, where $u(-L-1) = u(L+1) = 0$. We rewrite \eqref{e.ks5} as
\begin{equation}\label{e.ks6}
V_n = E - \frac{u(n+1)}{u(n)} - \frac{u(n-1)}{u(n)}.
\end{equation}
This motivates a change of variables,
\begin{equation}\label{e.ks7}
\tilde V = \{ V_n \}_{n = -L}^L \quad \longleftrightarrow \quad \{ x_{-L} , \ldots , x_{-1} , E , x_1 , \ldots , x_L \},
\end{equation}
where
$$
E = E_{\tilde V}^{L,k}
$$
and
$$
x_n = \begin{cases} \frac{\varphi_{\tilde V}^{L,k} (n+1)}{\varphi_{\tilde V}^{L,k} (n)} & n < 0 \\ \frac{\varphi_{\tilde V}^{L,k} (n-1)}{\varphi_{\tilde V}^{L,k} (n)} & n > 0, \end{cases}
$$
so that
$$
V_n = \begin{cases} E - x_{n-1}^{-1} - x_n & n < 0 \\ E - x_{-1}^{-1} - x_1^{-1} & n = 0 \\ E - x_{n+1}^{-1} - x_n & n > 0, \end{cases}
$$
with $x_{-L-1}^{-1} = x_{L+1}^{-1} = 0$ (which is natural in view of the definition above).

We wish to rewrite \eqref{e.ks8} using this change of variables. In order to do this, we have to determine the Jacobian of the change of variables \eqref{e.ks7}. We claim that it satisfies
\begin{align}\label{e.ks9}
\det J & = 1 + x_1^{-2} \left\{ 1 + x_2^{-2} \left\{ 1 + \cdots x_{L-1}^{-2} \left\{ 1 + x_L^{-2} \right\} \cdots \right\} \right\} \\
\nonumber & \qquad  + x_{-1}^{-2} \left\{ 1 + x_{-2}^{-2} \left\{ 1 + \cdots x_{-L+1}^{-2} \left\{ 1 + x_{-L}^{-2} \right\} \cdots \right\} \right\} \\
\nonumber & = \varphi_{\tilde V}^{L,k} (0)^{-2}.
\end{align}

To prove \eqref{e.ks9}, we note that
$$
J = \begin{pmatrix} -1 & x_{-L}^{-2} &&&&&&& \\ & -1 & x_{-L+1}^{-2} &&&&&& \\ && \ddots & \ddots &&&&& \\ &&& -1 & x_{-1}^{-2} &&&& \\ 1 & 1 & \cdots & 1 & 1 & 1 & \cdots & 1 & 1 \\ &&&& x_1^{-2} & - 1 &&& \\ &&&&& x_2^{-2} & -1 && \\ &&&&&& \ddots & \ddots & \\ &&&&&&& x_{L}^{-2} & -1 \end{pmatrix}
$$
where the row index runs from $x_{-L}$ at the top to $x_L$ at the bottom (with $E$ in the middle) and the column index runs from $V_{-L}$ on the left to $V_L$ on the right. Thus, if we expand along the central ($V_0$-)column, we find
\begin{align*}
\det J & = 1 \cdot \det \begin{pmatrix} -1 & \ddots &&&& \\ & \ddots & \ddots &&& \\ && -1 &&& \\ &&& -1 && \\ &&& \ddots & \ddots & \\ &&&& \ddots & -1 \end{pmatrix} \\
& \qquad - x_{-1}^{-2}  \cdot \det \begin{pmatrix} -1 & \ddots &&&&& \\ & \ddots & \ddots &&&& \\ && -1 &&&& \\ 1 & 1 & \cdots & 1 & \cdots & 1 & 1 \\ &&&& -1 && \\ &&&& \ddots & \ddots & \\ &&&&& \ddots & -1  \end{pmatrix} \\
& \qquad - x_{1}^{-2}  \cdot \det \begin{pmatrix} -1 & \ddots &&&&& \\ & \ddots & \ddots &&&& \\ && -1 &&&& \\ 1 & 1 & \cdots & 1 & \cdots & 1 & 1 \\ &&&& -1 && \\ &&&& \ddots & \ddots & \\ &&&&& \ddots & -1  \end{pmatrix}
\end{align*}
The first term is $1$ and the other two terms may be expanded along the central column again. Iterating the procedure, the first identity in \eqref{e.ks9} follows. The second identity follows from the definition of the $x_n$'s: We have
$$
1 + x_{-L}^{-2} = 1 + \frac{\varphi_{\tilde V}^{L,k} (-L)^2}{\varphi_{\tilde V}^{L,k} (-L+1)^2}
$$
and
\begin{align*}
1 + x_{-L+1}^{-2} \left( 1 + x_{-L}^{-2} \right) & = 1 + \frac{\varphi_{\tilde V}^{L,k} (-L+1)^2}{\varphi_{\tilde V}^{L,k} (-L+2)^2} \left( 1 + \frac{\varphi_{\tilde V}^{L,k} (-L)^2}{\varphi_{\tilde V}^{L,k} (-L+1)^2} \right) \\
& = 1 + \frac{\varphi_{\tilde V}^{L,k} (-L+1)^2}{\varphi_{\tilde V}^{L,k} (-L+2)^2} + \frac{\varphi_{\tilde V}^{L,k} (-L)^2}{\varphi_{\tilde V}^{L,k} (-L+2)^2}.
\end{align*}
Iterating this, we find
$$
x_{-1}^{-2} \left\{ 1 + x_{-2}^{-2} \left\{ 1 + \cdots x_{-L+1}^{-2} \left\{ 1 + x_{-L}^{-2} \right\} \cdots \right\} \right\} = \sum_{n=-1}^{-L} \frac{\varphi_{\tilde V}^{L,k} (n)^2}{\varphi_{\tilde V}^{L,k} (0)^2}.
$$
Similarly, it follows that
$$
x_1^{-2} \left\{ 1 + x_2^{-2} \left\{ 1 + \cdots x_{L-1}^{-2} \left\{ 1 + x_L^{-2} \right\} \cdots \right\} \right\} = \sum_{n=1}^{L} \frac{\varphi_{\tilde V}^{L,k} (n)^2}{\varphi_{\tilde V}^{L,k} (0)^2}.
$$
Thus,
\begin{align*}
\det J & = 1 + x_1^{-2} \left\{ 1 + x_2^{-2} \left\{ 1 + \cdots x_{L-1}^{-2} \left\{ 1 + x_L^{-2} \right\} \cdots \right\} \right\} \\
& \qquad  + x_{-1}^{-2} \left\{ 1 + x_{-2}^{-2} \left\{ 1 + \cdots x_{-L+1}^{-2} \left\{ 1 + x_{-L}^{-2} \right\} \cdots \right\} \right\} \\
& \frac{\varphi_{\tilde V}^{L,k} (0)^2}{\varphi_{\tilde V}^{L,k} (0)^2} + \sum_{n=-1}^{-L} \frac{\varphi_{\tilde V}^{L,k} (n)^2}{\varphi_{\tilde V}^{L,k} (0)^2} + \sum_{n=1}^{L} \frac{\varphi_{\tilde V}^{L,k} (n)^2}{\varphi_{\tilde V}^{L,k} (0)^2} \\
& = \frac{\|\varphi_{\tilde V}^{L,k}\|^2}{\varphi_{\tilde V}^{L,k} (0)^2} \\
& = \varphi_{\tilde V}^{L,k} (0)^{-2},
\end{align*}
since $\varphi_{\tilde V}^{L,k}$ is a normalized eigenvector. This proves \eqref{e.ks9}.

We also note the following,
\begin{equation}\label{e.ks10}
\left| \varphi_{\tilde V}^{L,k} (m) \right| \left| \varphi_{\tilde V}^{L,k} (0) \right|^{-1} = \left| x_1^{-1} \cdots x_m^{-1} \right|.
\end{equation}

We are now ready to carry out the substitution \eqref{e.ks7} in the formula \eqref{e.ks8} for $\varrho_L(m,0)$, using the identities \eqref{e.ks9} and \eqref{e.ks10}:
\begin{align*}
\varrho_L(m,0) & = \int \cdots \int \left( \sum_k \left| \left\langle \delta_m , \varphi_{\tilde V}^{L,k} \right\rangle \right| \, \left| \left\langle \delta_0 , \varphi_{\tilde V}^{L,k} \right\rangle \right| \right) \prod_{n = -L}^L r(V_n) \, dV_{-L} \cdots dV_L \\
& = \sum_k \int \cdots \int \left| \varphi_{\tilde V}^{L,k}(m) \right| \, \left| \varphi_{\tilde V}^{L,k} (0) \right| \prod_{n = -L}^L r(V_n) \, dV_{-L} \cdots dV_L \\
& = \sum_k \int \cdots \int \left| \varphi_{\tilde V}^{L,k}(m) \right| \, \left| \varphi_{\tilde V}^{L,k} (0) \right|^{-1} \prod_{n = -L}^L r(V_n) \left| \varphi_{\tilde V}^{L,k} (0) \right|^2 \, dV_{-L} \cdots dV_L \\
& = \int_{\Sigma_0} \int_{\R^{2L}} \left| x_1^{-1} \cdots x_m^{-1} \right| \left( \prod_{n=-1}^{-L} r (E - x_{n-1}^{-1} - x_n) \right) r(E - x_1^{-1} - x_{-1}^{-1}) \\
& \qquad \times \left( \prod_{n=1}^{L} r (E - x_{n+1}^{-1} - x_n) \right) \, dx_{-L} \cdots dx_{-1} \, dx_1 \cdots dx_L \, dE.
\end{align*}

For fixed $E \in \Sigma_0$, it follows from the definitions that the inner integral has the required form:
\begin{align*}
& \left\langle T_1^{m-1} T_0^{L-m} \phi , U T_0^L \phi \right\rangle_{L^2(\R)} = \int_{\R^{2L}} \left| x_1^{-1} \cdots x_m^{-1} \right| \left( \prod_{n=-1}^{-L} r (E - x_{n-1}^{-1} - x_n) \right) \\
& \qquad \qquad \times r(E - x_1^{-1} - x_{-1}^{-1}) \left( \prod_{n=1}^{L} r (E - x_{n+1}^{-1} - x_n) \right) \, dx_{-L} \cdots dx_{-1} \, dx_1 \cdots dx_L
\end{align*}
The formula for $\varrho_L(m,0)$ claimed in the lemma therefore follows.
\end{proof}

In the following lemma, we denote the norm of an operator $T : L^p(\R) \to L^q(\R)$ by $\|T\|_{p,q}$.

\begin{lemma}\label{l.ks4}
{\rm (a)} $\|T_0\|_{1,1} \le 1$. \\
{\rm (b)} $\sup \{ \|T_0\|_{1,2} : E \in \Sigma_0 \} \le \|r\|_\infty^{1/2} < \infty$. \\
{\rm (c)} $\|T_1\|_{2,2} \le 1$. \\
{\rm (d)} $T_1^2$ is compact. \\
{\rm (e)} $\sup \{ \|T_1^2\|_{2,2} : E \in \Sigma_0 \} < 1$.
\end{lemma}

\begin{proof}
(a) For $f \in L^1(\R)$, we have
\begin{align*}
\|T_0 f\|_1 & = \int |T_0f(x)| \, dx \\
& = \int \left| \int r(E - x - y^{-1}) f(y) \, dy \right| \, dx \\
& \le \iint | r(E - x - y^{-1}) | \, | f(y) | \, dy \, dx \\
& = \int \left( \int r(E - x - y^{-1}) \, dx \right) | f(y) | \, dy \\
& = \int \left( \int r(\tilde x) \, d\tilde x \right) | f(y) | \, dy \\
& = \int | f(y) | \, dy \\
& = \|f\|_1.
\end{align*}
Here we used that $r$ is the (non-negative) density of an absolutely continuous probability measure.

(b) For $f \in L^1(\R)$, we have
\begin{align*}
\|T_0 f\|_2^2 & = \int |T_0f(x)|^2 \, dx \\
& = \int \left| \int r(E - x - y^{-1}) f(y) \, dy \right|^2 \, dx \\
& \le \|r\|_\infty \|f\|_1 \iint | r(E - x - y^{-1}) | \, | f(y) | \, dy \, dx \\
& = \|r\|_\infty \|f\|_1^2.
\end{align*}
Here we use an identity from the proof of (a) in the last step.

(c) We have
\begin{align*}
T_1 f(x) & = \int r(E - x - y^{-1}) |y|^{-1} f(y) \, dy \\
& = \int r(E - x - y^{-1}) |y| f(y) |y|^{-2} \, dy \\
& = \int r(E - x + \tilde y) |\tilde y|^{-1} f(- (\tilde y)^{-1}) \, d\tilde y.
\end{align*}
Here we used the substitution $y = -(\tilde y)^{-1}$. Thus, we can write $T_1 = K \tilde U$, where
$$
\tilde Uf(x) = |x|^{-1} f(-x^{-1})
$$
and
$$
Kf(x) = \int r(E - x + y) f(y) \, dy.
$$
Since we can write $Kf = r_E \ast f$ with $r_E(x) = r(E - x)$, it follows that
\begin{align*}
\|T_1 f\|_2 & = \|K \tilde U f\|_2 \\
& = \|r_E \ast \tilde U f\|_2 \\
& = \|\widehat{r_E \ast \tilde U f}\|_2 \\
& = \|\widehat{r_E} \cdot \widehat{\tilde U f}\|_2 \\
& \le \| \widehat{r_E} \|_\infty \cdot \| \widehat{\tilde U f}\|_2 \\
& = \|r_E \|_1 \cdot \| \tilde U f\|_2 \\
& = \|f\|_2.
\end{align*}
In the last step, we used that
$$
\| \tilde U f\|_2^2 = \int |x|^{-2} |f(-x^{-1})|^2 \, dx = \int |f(\tilde x)|^2 \, d\tilde x = \|f\|_2^2.
$$

(d) Denote the Fourier transform by $F$ and define $\hat K = FKF^{-1}$, $\hat U = F \tilde U F^{-1}$. Then,
$$
T_1^2 = (K \tilde U)^2 = K \tilde U K \tilde U = F^{-1} \hat K \hat U \hat K \hat U F.
$$
Thus, to show that $T_1^2$ is compact, it suffices to show that $\hat K \hat U \hat K$ is compact. To do so, we will show that this operator has an $L^2$ integral kernel and hence is Hilbert-Schmidt.

Let $g_1 \in C_0^\infty(\R)$ which is $\equiv 1$ in a neighborhood of $0$ and let $g_2 = 1 - g_1$. Set
$$
U_i f(x) = g_i (x) \cdot \tilde U f(x) , \quad i = 1,2.
$$
In particular, $\tilde U f = U_1 f + U_2 f$.

We have
\begin{align*}
\widehat{U_1 f} (k) & = \frac{1}{\sqrt{2\pi}} \int e^{-ikx} U_1 f(x) \, dx \\
& = \frac{1}{\sqrt{2\pi}} \int e^{-ikx} g_1 (x) |x|^{-1} f(-x^{-1}) \, dx \\
& = \frac{1}{\sqrt{2\pi}} \int e^{-ik\tilde x^{-1}} g_1 (\tilde x^{-1}) |\tilde x|^{-1} f(-\tilde x) \, d\tilde x \\
& = \frac{1}{2\pi} \int e^{-ik\tilde x^{-1}} g_1 (\tilde x^{-1}) |\tilde x|^{-1} \left( \int e^{-ip\tilde x} \hat f(p) \, dp \right) \, d\tilde x \\
& = \frac{1}{2\pi} \int \left( \int e^{-ik\tilde x^{-1}-ip\tilde x} g_1 (\tilde x^{-1}) |\tilde x|^{-1} \, d\tilde x \right) \hat f(p) \, dp \\
& = \frac{1}{2\pi} \int \left( \int e^{-ikx-ipx^{-1}} g_1 (x) |x|^{-1} \, dx \right) \hat f(p) \, dp \\
& =: \frac{1}{2\pi} \int a_1(k,p) \hat f(p) \, dp
\end{align*}

Similarly,
$$
\widehat{U_2 f} (k) = \frac{1}{2\pi} \int a_2(k,p) \hat f(p) \, dp
$$
with
$$
a_2(k,p) = \int e^{-ikx-ipx^{-1}} g_2 (x) |x|^{-1} \, dx.
$$

The integral kernel of $\hat K \hat U \hat K$ is therefore given by
$$
b(k,p) = \hat r_E(k) a_1 (k,p) \hat r_E(p) + \hat r_E(k) a_2 (k,p) \hat r_E(p),
$$
so that
\begin{align*}
\| b \|_{L^2(\R,dk) \times L^2(\R,dp)} & \le \| \hat r \|_{L^2(\R,dk)} \sup_k \|a_1 (k,\cdot)\|_{L^2(\R,dp)} \|\hat r\|_{L^\infty(\R,dp)} \\
& + \| \hat r \|_{L^\infty(\R,dk)} \sup_p \|a_2 (\cdot,p)\|_{L^2(\R,dp)} \|\hat r\|_{L^2(\R,dp)}.
\end{align*}

Note that $r \in L^1(\R) \cap L^\infty(\R)$ implies that $r \in L^2 (\R) \cap L^1(\R)$, and therefore $\hat r \in L^2(\R) \cap L^\infty(\R)$. Thus, it remains to show that
\begin{equation}\label{e.ks3}
\sup_k \|a_1 (k,\cdot)\|_{L^2(\R,dp)} < \infty
\end{equation}
and
\begin{equation}\label{e.ks4}
\sup_p \|a_2 (\cdot,p)\|_{L^2(\R,dp)} < \infty.
\end{equation}

For fixed $k$, $a_1(k,\cdot)$ has a potential problem near $0$, but
$$
\int_{|x| > \frac{1}{N}} e^{-i(kx + px^{-1})} \frac{g_1(x)}{|x|} \, dx = \int_{|\tilde x| < N} e^{-i(k\tilde x^{-1}+p\tilde x)} \frac{g_1(\tilde x^{-1})}{|\tilde x|} \, d\tilde x,
$$
which converges, as $N \to \infty$, in $L^2$-sense to the Fourier transform of the $L^2$-function $e^{-i(k\tilde x^{-1}+p\tilde x)} \frac{g_1(\tilde x^{-1})}{|\tilde x|}$, whose $L^2$-norm is independent of $k$. Thus, \eqref{e.ks3} follows.

For fixed $p$, $a_2(\cdot,p)$ is by definition the Fourier transform of the function
$$
f^{(p)}(x) = \sqrt{2\pi} \, e^{-ipx^{-1}} \frac{g_2(x)}{|x|}.
$$
Since $g_2$ vanishes near $0$, this is an $L^2$ function. Obviously, its $L^2$ norm is independent of $p$. By unitarity of the Fourier transform, \eqref{e.ks4} follows.

(e) Since $T_1^2$ is compact, so is $|T_1^2| = ((T_1^2)^* T_1^2)^{1/2}$. Since $|T_1^2|$ is also positive, $\|T_1^2\|$ is an eigenvalue of $|T_1^2|$. On the other hand, recall from the proof of (c) that $\|T_1 f\|_2 = \|\widehat{r_E} \cdot \widehat{\tilde U f}\|_2$. Since $|\widehat{r_E}(k)| < 1$ for $k \not= 0$, we have that $\|T_1 f\|_2 < \|\widehat{\tilde U f}\|_2 = \|f\|_2$ for every non-zero $f \in L^2(\R)$. Thus, $1$ is not an eigenvalue of $|T_1^2|$. Thus, $\|T_1^2\| \not= 1$. By (c), we therefore have $\|T_1^2\| < 1$. Finally, $\|T_1^2\|$ is independent of $E$ since $\widehat{r_E}(k) = e^{-iEk} \widehat{r_0}(k)$ and hence multiplication by $\widehat{r_E}$ has norm independent of $E$.
\end{proof}

\begin{proof}[Proof of Theorem~\ref{t.ksd}.]
Since $\mu$ is $T$-invariant, it suffices to prove \eqref{e.ks1} for $m \ge 0$ and $n = 0$. We have
\begin{align*}
\int_\Omega & \left( \sup_{t \in \R} \left| \left\langle \delta_m , e^{-itH_\omega} \delta_0 \right\rangle \right| \right) \, d\mu(\omega) \\
& = a(m,0) \\
& \le \limsup_{L \to \infty} a_L(m,0) \\
& \le \limsup_{L \to \infty} \varrho_L(m,0) \\
& = \limsup_{L \to \infty} \int_{\Sigma_0} \left\langle T_1^{m-1} T_0^{L-m} \phi , U T_0^L \phi \right\rangle_{L^2(\R)} \, dE \\
& \le \limsup_{L \to \infty} \int_{\Sigma_0} \left\| T_1^{m-1} T_0^{L-m} \phi \right\|_2 \left\| T_0^L \phi \right\|_2 \, dE \\
& \le \limsup_{L \to \infty} \int_{\Sigma_0} \| T_1^{m-1} \|_{2,2} \, \|T_0\|_{1,2} \, \|T_0^{L-m-1}\|_{1,1} \, \|\phi\|_1 \cdot \|T_0\|_{1,2} \, \| T_0^{L-1} \|_{1,1} \, \|\phi \|_1 \, dE \\
& \le \sup \{ \|T_1^2\|_{2,2} : E \in \Sigma_0 \}^{\frac{m-2}{2}} \cdot \|r\|_\infty \cdot \mathrm{Leb}(\Sigma_0).
\end{align*}
The first step is just the definition of $a(m,0)$, the second step follows from Lemma~\ref{l.ks1}, the third step follows from Lemma~\ref{l.ks2}, the fourth step follows from Lemma~\ref{l.ks3}, the fifth step follows from Cauchy-Schwarz and unitarity of $U$, the sixth step is obvious, the seventh step follows from Lemma~\ref{l.ks4} and $\|\phi\|_1 = 1$. Since $\sup \{ \|T_1^2\|_{2,2} : E \in \Sigma_0 \} < 1$ by Lemma~\ref{l.ks4}, the theorem follows.
\end{proof}

We mentioned earlier that dynamical localization implies spectral localization. Here is how to derive a spectral localization result from the dynamical localization result contained in Theorem~\ref{t.ksd}.

\begin{prop}
If there are constants $C, \gamma \in (0,\infty)$ such that
$$
\max_{n \in \{0,1\}} a(m,n) \le C e^{-\gamma |m|},
$$
with $a(m,n)$ as in \eqref{e.amndef}, then for $\mu$-almost every $\omega \in \Omega$, $H_\omega$ has pure point spectrum with exponentially decaying eigenvectors. More precisely, these eigenvectors obey estimates of the form
$$
|u(n)| \le C_{\omega,\varepsilon,u} e^{-(\gamma - \varepsilon)|n|}
$$
for small $\varepsilon > 0$.
\end{prop}

\begin{proof}
Let
$$
a(m,n; \omega) = \sup_{t \in \R} \left| \left\langle \delta_m , e^{-itH_\omega} \delta_n \right\rangle \right|,
$$
so that $a(m,n) = \int a(m,n;\omega) \, d\mu(\omega)$. By the Chebyshev inequality, we have for $\varepsilon > 0$,
\begin{equation}\label{e.ks11}
\mu \left( \omega : a(m,n;\omega) > e^{-(\gamma-\varepsilon)|m|} \right) \le e^{(\gamma - \varepsilon)|m|} a(m,n).
\end{equation}
By assumption, for $n \in \{ 0,1 \}$, $a(m,n) \le C e^{-\gamma |m|}$ and hence the left-hand side of \eqref{e.ks11} is summable, so that by Borel-Cantelli, we have
$$
\mu \left( \omega : a(m,n;\omega) > e^{-(\gamma-\varepsilon)|m|} \text{ for infinitely many } n \right) = 0
$$
for $n \in \{0,1\}$.

Let us consider the full-measure set of $\omega$'s for which $a(m,n;\omega) \le \tilde C_{\omega,\varepsilon} e^{-(\gamma - \varepsilon)|m|}$ for $n \in \{ 0,1 \}$ and all $m \in \Z$. By the RAGE Theorem and cyclicity of $\{ \delta_0 , \delta_1 \}$, $H_\omega$ has pure point spectrum for such $\omega$'s.

To prove exponential decay of the corresponding eigenvectors, consider the functions
$$
f_T(x) = \frac{1}{T} \int_0^T e^{isE} e^{-isx} \, ds.
$$
We have $|f_T(x)| \le 1$ and $f_T(x) \to 0$ (resp., $1$) as $T \to \infty$ for $x \not= E$ (resp., $x = E$). Thus, by the functional calculus,
\begin{equation}\label{e.ks12}
\chi_{\{E\}}(H_\omega) = \mathrm{s}-\lim_{T \to \infty} \frac{1}{T} \int_0^T e^{isE} e^{-isH_\omega} \, ds.
\end{equation}
Since in one dimension all eigenvalues of Schr\"odinger operators are simple, the normalized eigenvector $u_{\omega,E}$ corresponding to the eigenvalue $E$ of $H_\omega$ obeys
\begin{align*}
|u_{\omega,E}(m)| & = |u_{\omega,E}(0)|^{-1} | \langle \delta_0 , u_{\omega,E} \rangle \langle \delta_m , u_{\omega,E} \rangle | \\
& = |u_{\omega,E}(0)|^{-1} | \langle \delta_m , \chi_{\{E\}} (H_\omega) \delta_0 \rangle | \\
& \le |u_{\omega,E}(0)|^{-1} a(m,0;\omega) \\
& \le C_{\omega,\varepsilon,u} e^{-(\gamma - \varepsilon)|m|},
\end{align*}
provided that $u_{\omega,E}(0) \not= 0$. We used \eqref{e.ks12} in the third step and our choice of $\omega$ in the fourth step. In the case where $u_{\omega,E}(0) = 0$, we must have $u_{\omega,E}(1) \not= 0$ since $u_{\omega,E}$ is a non-trivial solution of \eqref{e.eved}, and we repeat the steps above with $1$ in place of $0$.
\end{proof}

\subsection{The Continuum Case}\label{ss.ksc}

See \cite{ds, r}.

\end{document}